\documentclass[12pt]{amsart}
\usepackage{amssymb}
\usepackage{amsmath}
\usepackage{amsthm}
\usepackage{amsbsy}
\usepackage{psfrag}
\usepackage{graphics}
\usepackage{graphicx}

\usepackage{color}
\theoremstyle{plain}
\newtheorem{theorem}{Theorem}[section]
\newtheorem{lemma}[theorem]{Lemma}
\newtheorem{corollary}[theorem]{Corollary}

\theoremstyle{remark}
\newtheorem{remark}[theorem]{Remark}
\begin{document}
\def\O{\Omega}
\def\p{\partial}
\def\cN{\mathcal{N}}
\def\tQ{\tilde{Q}}
\def\tu{{\tilde u}}
\def\Pij{\Pi_j}
\def\bQ{\mathbb{Q}}
\def\R{\mathbb{R}}
\def\d{\displaystyle}
\def\cV{\mathcal{V}}
\def\dotleq{\hspace{4pt}\raise 1pt\hbox{$\leq$}\hspace{-6pt}\lower 4pt\hbox{.}\hspace{6pt}}
\def\tK{{\tilde K}}
\def\Date{December 12, 2012}
\allowdisplaybreaks[4]
\numberwithin{equation}{section} \numberwithin{figure}{section}
\numberwithin{table}{section}
\title[A Generalized Finite Element Method for the Obstacle Problem of Plates]
 {A Generalized Finite element method\\
 for the displacement obstacle problem of\\ Clamped Kirchhoff plates}
\thanks{The work of the first and third authors was supported in part
by the National Science Foundation under Grant No.
 DMS-10-16332.  The work of the second author was supported in part by
 the National Science Foundation through
 the VIGRE Grant 07-39382.}
\author{Susanne C. Brenner}
\address{Department of Mathematics and Center for Computation and Technology,
 Louisiana State University, Baton Rouge, LA 70803}
\email{brenner@math.lsu.edu}
\author{Christopher B. Davis}
\address{Department of Mathematics and Center for Computation and Technology,
 Louisiana State University, Baton Rouge, LA 70803}
 \email{cdav135@lsu.edu}
\author{Li-yeng Sung}
\address{Department of Mathematics and Center for Computation and Technology,
 Louisiana State University, Baton Rouge, LA 70803}
\email{sung@math.lsu.edu}
\begin{abstract} A generalized finite element method for the displacement obstacle
 problem of clamped Kirchhoff plates is considered in this paper.  We derive optimal error estimates and present
 numerical results that illustrate the performance of the method.
\end{abstract}
%
%
\maketitle
%
\section{Introduction}\label{sec:Introduction}
 Let $\O$ be a bounded polygonal domain $\Omega\subset\mathbb{R}^2$, $f\in L_2(\O)$,
 $g\in H^4(\O)$, and $\psi_1,\psi_2\in C^2(\O)\cap C(\bar\O)$ be two obstacle functions such that
\begin{equation}\label{eq:ObstacleCondition}
 \psi_1<\psi_2\;\;\text{in}\;\;\O\quad\text{and}\quad \psi_1<g<\psi_2\;\;\text{on}\;\p\O.
\end{equation}
 Consider the following problem: Find $u\in H^2(\O)$ such that
\begin{equation}\label{eq:ObstacleProblem}
  u=\mathop{\rm argmin}_{v\in K}G(v),
\end{equation}
 where
\begin{align}
 K&=\{v\in H^2(\O):\,u-g\in H^2_0(\O),\;\psi_1\leq v\leq \psi_2\;\text{on}\;\O\},\label{eq:KDef}\\
 G(v)&=\frac12 a(v,v)-(f,v),\label{eq:GDef}\\
 a(v,w)&=\int_\O \nabla^2v:\nabla^2 w\,dx,\quad (f,v)=\int_\O fv\,dx\label{eq:aDef}
\end{align}
 and $\nabla^2 v:\nabla^2 w=\sum_{i,j=1}^2 v_{x_ix_j}w_{x_ix_j}$ is the
 (Frobenius) inner product of the Hessian matrices of $v$ and $w$.
\par
 Since $K$ is a nonempty closed convex subset of $H^2(\O)$ and $a(\cdot,\cdot)$ is symmetric
 and coercive on
  $H^2_0(\O)$ which contains the set $K-K$, it follows from the standard theory
  \cite{LS:1967:VarInequalities,GLT:1981:VI,KS:1980:VarInequalities,Friedman:1988:VI} that
   \eqref{eq:ObstacleProblem} has
  a unique solution $u\in K$ characterized by the following variational inequality:
\begin{equation}\label{eq:VariationalInequality}
  a(u,v-u)\geq (f,v-u)\qquad\forall\,v\in K.
\end{equation}
\par
 The convergence of finite element methods for second order obstacle problems was
 investigated in \cite{Falk:1974:VarInequality,BHR:1977:VI,BHR:1978:VI}, shortly after
 it was shown in \cite{BS:1968:VI} that the solutions for such obstacle problems
 belong to $H^2(\O)$ under appropriate regularity assumptions on the data.
 This full elliptic regularity allows the complementarity form of the
 variational inequality (in the strong sense) to be used in the convergence analysis.
\par
 In contrast, it was shown in
 \cite{Frehse:1971:VarInequality,Frehse:1973:VI,CF:1979:BiharmonicObstacle}
 that the solution $u$ of \eqref{eq:ObstacleProblem}/\eqref{eq:VariationalInequality}
 belongs to $H^3_{loc}(\O)\cap C^2(\O)$ under the assumptions above on $f$, $g$, $\psi_1$ and $\psi_2$.
 Since the obstacles are separated from each other and from the displacement boundary condition
(cf. \eqref{eq:ObstacleCondition}),
 we have $\Delta^2 u=f$ near $\p\O$.  Therefore it follows from
 the elliptic regularity theory for the biharmonic operator on polygonal domains
 \cite{BR:1980:Biharmonic,Grisvard:1985:EPN,Dauge:1988:EBV,KMR:2001:Corner} that
 $u\in H^{2+\alpha}(\cN)$  for some $\alpha\in (\frac12,1]$ in an open neighborhood $\cN$ of $\p\O$.
   The elliptic regularity index $\alpha$ is determined by
 the interior angles of $\O$ and we can take $\alpha$ to be $1$ for convex $\O$.  Thus the solution
 $u$ of \eqref{eq:ObstacleProblem}/\eqref{eq:VariationalInequality} belongs to
 $H^{2+\alpha}(\O)\cap H^3_{loc}(\O)\cap C^2(\O)$ in general.  Moreover, it is easy to construct examples where
 $u\notin H^4_{loc}(\O)$ even for smooth data \cite{CF:1979:BiharmonicObstacle}.
\par
 This lack of $H^4_{loc}(\O)$ regularity means that the complementarity form of \eqref{eq:VariationalInequality}
 only exists in a weak sense \cite{CF:1979:BiharmonicObstacle}.  Consequently convergence
  analysis based on
 the weak complementarity form of \eqref{eq:VariationalInequality} would only lead to suboptimal error
 estimates.
\par
 A new convergence analysis for finite element methods for
 \eqref{eq:ObstacleProblem}/\eqref{eq:VariationalInequality} that does not
  rely on the complementarity form of the variational inequality
 \eqref{eq:VariationalInequality} was proposed in \cite{BSZ:2012:Obstacle}, where optimal
 convergence was established for $C^1$ finite element methods, classical nonconforming finite element
 methods, and $C^0$ interior penalty methods for clamped plates ($g=0$) on convex domains.
 The results in \cite{BSZ:2012:Obstacle} were subsequently extended to general polygonal domains
 and general Dirichlet boundary conditions for a quadratic $C^0$ interior penalty method
 \cite{BSZZ:2012:Kirchhoff} and  a Morley finite element method
 \cite{BSZZ:2011:Morley}.  The goal of this paper is to extend the results in
 \cite{BSZZ:2012:Kirchhoff,BSZZ:2011:Morley} to a generalized finite element method
  for plates \cite{Davis:2011:Thesis,ODJ:2011:Plate}.
\par
 The rest of the paper is organized as follows.  We introduce the generalized finite element method
 in Section~\ref{sec:GFEM} and carry out the convergence analysis in Section~\ref{sec:Convergence}.
 Numerical results are reported in Section~\ref{sec:Numerics}.
%
\section{A Generalized Finite Element Method}\label{sec:GFEM}
 We begin with the
 construction of the approximation space $V_h$ in Section~\ref{subsec:ApproximationSpace} and define
 an interpolation
 operator from $H^2(\O)$ into $V_h$ in Section~\ref{subsec:Interpolation}.  The discrete obstacle
 problem is given in
 Section~\ref{subsec:DiscreteObstacle}.  We refer the readers to
 \cite{BBO:2003:GFEM,BB:2012:SGFEM} for various aspects of
 generalized finite element methods.
\subsection{Construction of the approximation space}\label{subsec:ApproximationSpace}
 The approximation space is based on partition of unity by flat-top functions
 \cite{MB:1996:PTU,OKH:2008:PTU}.
\subsubsection{Partition of Unity}
 Let $\phi$ be the $C^1$ piecewise polynomial function given by
\begin{equation*}
        \phi(x) =\left\{
        \begin{array}{ll}
            \phi^L(x):= (1+x)^2 (1-2x)  &\mbox{ if } x \in [-1,   0] \\
            \phi^R(x):= (1-x)^2 (1+2x) &\mbox{ if } x \in [0,   1]   \\
            0               &\mbox{ if } |x| \ge 1
        \end{array}\right. ,
\end{equation*}
 which enjoys the partition of unity property that
\begin{equation}\label{eq:PU0}
  \phi^L(x-1)+\phi^R(x)=1\qquad\text{for}\quad 0\leq x\leq 1.
\end{equation}
 We define a flat-top function $\psi_\delta$ by
\begin{equation*}
    \psi_\delta(x)=\left\{
    \begin{array}{ll}
        \phi^L\left(\frac{x-(-1+\delta)}{2\delta}\right)& \mbox{ if }x   \in[-1-\delta, -1+\delta] \\
        1                                                    & \mbox{ if }x   \in[-1+\delta, 1-\delta] \\
            \phi^R\left(\frac{x-(1-\delta)}{2\delta}\right)& \mbox{ if }x   \in[1-\delta, 1+\delta] \\
            0                                    & \mbox{ if }x\notin[-1-\delta, 1+\delta] \\
        \end{array}\right. .
\end{equation*}
 Here $\delta$ is a small number that controls the width of
 the flat-top part of this function where $\psi_\delta=1$.
\par
 For ease of presentation we take $\O$ to be a  rectangle $(a,b)\times(c,d)$.
 But the construction and analysis can be extended to
 other domains (cf. Remark~\ref{rem:LShaped}  and Examples~4 and 5 in
 Section~\ref{sec:Numerics}).
\par
 We first expand $\Omega$ to a larger rectangle $\tilde\O=(a-\gamma_1,b+\gamma_1)\times (c-\gamma_2,d+\gamma_2)$
 where $\gamma_1$ and $\gamma_2$ are two positive numbers, and then we divide
 $\tilde\O$ into disjoint congruent closed rectangular patches $Q_j$
 (cf. Figure~\ref{fig:Discretize}) with center $y_j=(y_{j,1},y_{j,2})$,
 width $h_1$ and height $h_2$,
 for $j=1,\ldots,N$.  We assume that the numbers
  $$\delta_j=\gamma_j/(h_j/2)\quad (j=1,2)$$
 belong to the interval $[\beta_1,\beta_2]$, where $\beta_1$ and $\beta_2$ are constants that
 satisfy $0<\beta_1<\beta_2<1$.
\par
  For each patch $Q_j$, let
\begin{equation*}
        \Psi_{j}(x)=\psi_{\delta_1}\Big(\frac{x_1-y_{j,1}}{h_1/2}\Big)
        \psi_{\delta_2}\Big(\frac{x_2-y_{j,2}}{h_2/2}\Big).
\end{equation*}
 It follows from \eqref{eq:PU0} that $\{ \Psi_{j}, j=1,\dots, N \}$ is a
 partition of unity in $\Omega$, i.e.,
 $$\sum_{j=1}^N\Psi_j=1\qquad\text{on}\quad\O.$$
  The flat-top region of each patch, defined by
\begin{equation*}
        Q_j^{\text{flat}}=\{x \in Q_j:\Psi_{j}(x)=1 \},
\end{equation*}
 is the rectangle
 centered at $y_j$ with width $h_1(1-\delta_1)=h_1-2\gamma_1$ and height
 $h_2(1-\delta_2)=h_2-2\gamma_2$ (cf. Figure~\ref{fig:Discretize}).
\goodbreak
\begin{remark}\label{rem:Patches}
 By construction we have (cf. Figure~\ref{fig:Discretize})
\begin{itemize}
 \item $Q_j^{\text{flat}} \cap Q_i^{\text{flat}} = \emptyset$ if $i \neq j.$
 \item The support of $ \Psi_j $ extends a horizontal distance of $\gamma_1=\delta_1(h/2)$ and a vertical distance
 of $\gamma_2=\delta_2(h/2)$ outside of the patch $Q_j$.  Hence
 the supports for $\Psi_i$ and $\Psi_j$ will intersect in a rectangular region of width
 $2\gamma_1$ or $2\gamma_2$ if $Q_i$ is a neighbor of $Q_j$.
 \item If $Q_j \cap \partial \Omega \neq \emptyset$, then $Q_j^{\text{flat}} \cap \partial
 \Omega \neq \emptyset.$
\end{itemize}
\end{remark}
\begin{figure}[h]
 \begin{center}
 \includegraphics[scale=1]{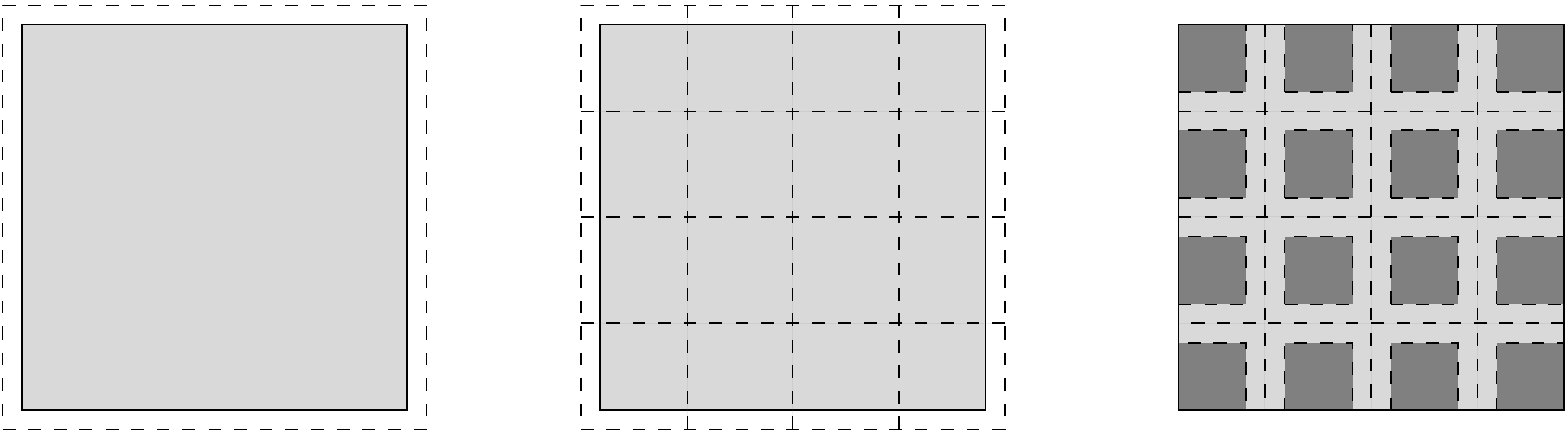}
\end{center}
\caption{Partitioning of a square domain $\Omega$:
 First we expand the shaded region $\Omega$ to $\tilde{\Omega}$ (left).  Next we
 divide $\tilde{\Omega}$ into congruent rectangles $Q_j$, $1\leq j \leq 16$ (middle).
 The flat top regions $Q_j^{\text{flat}}$, $1\leq j \leq 16$ are shown as the darker shaded regions
(right).}
\label{fig:Discretize}
\end{figure}
\subsubsection{Approximation space}
 The space $\bQ_2$ of biquadratic polynomials will serve as the local approximation space and the
 global approximation space is defined to be
\begin{equation*}
  V_h=\{\sum_{j=1}^N p_j\Psi_j:\,p_j\in \bQ_2\}.
\end{equation*}
 Below we present an explicit basis of $V_h$ that will be used in our numerical computations.
\par
  On the reference interval $[-1,1]$ we have two types of quadratic polynomials:
\begin{itemize}
 \item Lagrange interpolation polynomials $L_i(\xi)$ that satisfy  $N_i(L_j)=\delta_{ij}$
 for $1\leq i,j\leq 3$, where
  $N_1(v)=v(-1)$, $N_2(v)=v(0)$, and $N_3(v)=v(1)$.
\item Hermite interpolation polynomials $H_i(\xi)$ that satisfy  $N_i(H_j)=\delta_{ij}$ for
 $1\leq i,j\leq 3$, where
 $N_1(v)=v'(-1)$, $N_2(v)=v(-1)$, and $N_3(v)=v(1)$.
\end{itemize}
 The tensor product of different combinations of these polynomials will provide local bases
 on the two-dimensional rectangular patches.
\par
 Let $T_j:\mathbb{R}^2\longrightarrow \mathbb{R}^2$ be defined by
\begin{equation*}
 T_j(\xi_1,\xi_2)=(y_{j,1}+\xi_1(h_1/2)(1-\delta_1),y_{j,2}+\xi_2(h_2/2)(1-\delta_2)).
\end{equation*}
 Then $T_j$ maps the reference square $[-1,1]\times[-1,1]$ to the flat-top region $Q_j^{\text{flat}}$.
\par
 Depending on the location of the patch $Q_j$, we use different reference basis functions.  There
 are three possibilities.
\begin{itemize}\itemsep=5pt
 \item  For those patches such that $Q_j \cap \partial \Omega = \emptyset$, the reference
 basis functions are
\begin{equation}\label{innerpatch}
 \hat{f}_{ji}(\xi)=L_k(\xi_1)L_l(\xi_2),\quad i=3(l-1)+k,\quad 1\leq k,l \leq 3.
\end{equation}
 \item For those patches such that $Q_j$ intersects the boundary on only one side, say the vertical
 edge $x_1=a$ of $\O$,
 the reference basis functions are
\begin{equation}\label{sidepatch}
 \hat{f}_{ji}(\xi)=H_k(\xi_1)L_l(\xi_2),\quad i=3(l-1)+k,\quad 1\leq k,l \leq 3.
\end{equation}
 Note that in this case $T_j$ maps the line $\xi_1=-1$ to
 the part of $Q_j$ that intersects $\partial \Omega$.  The cases where $Q_j$ intersects other sides of
 $\O$ can be treated analogously.
\item For those patches such that $Q_j$
 intersects a corner of $\O$, say the lower left corner $(a,c)$, the reference
 basis functions are
\begin{equation}\label{cornerpatch}
 \hat{f}_{ji}(\xi)=H_k(\xi_1)H_l(\xi_2),\quad i=3(l-1)+k,\quad 1\leq k,l \leq 3.
\end{equation}
 Note that in this case $T_j$ maps the corner $(-1,-1)$ of the reference square to the
 lower left corner $(a,c)$ of $\O$.  The cases where $Q_j$ intersects other corners of $\O$ can be
 treated analogously.
\end{itemize}
\par
 The nodal variables (or degrees of freedom) for the local approximation space are depicted in
 Figure~\ref{fig:SquareRef}, where pointwise evaluations of functions, directional derivatives, gradients
 and mixed second order derivatives are represented by solid dots, arrows, circles and double arrows
 respectively.
\begin{figure}[h]
\begin{center}
\includegraphics[scale=1]{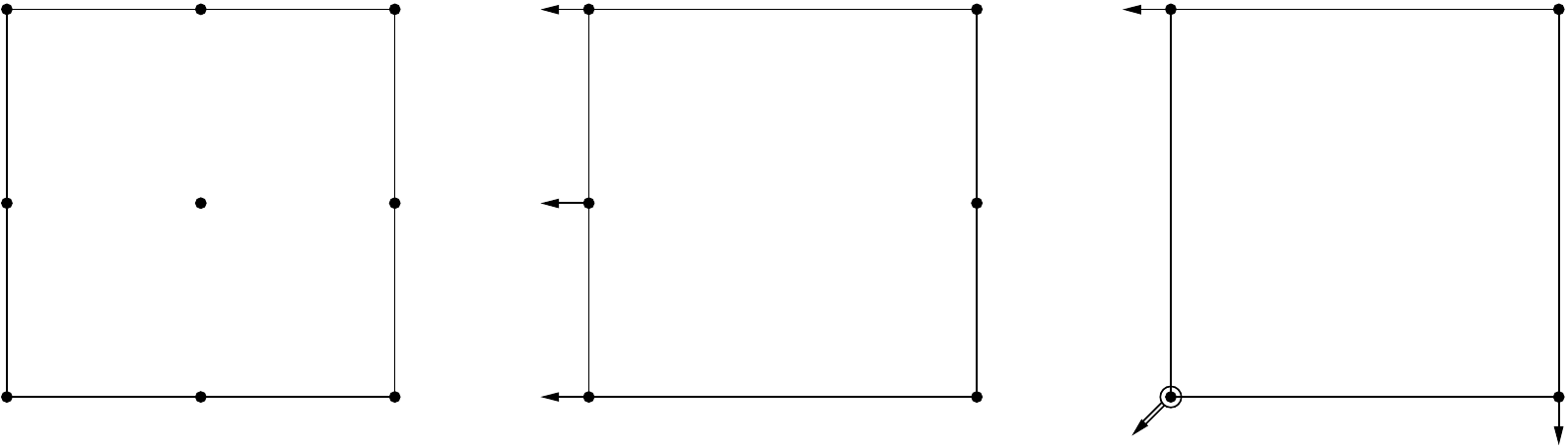}
 \caption{Reference elements described in (\ref{innerpatch}),
(\ref{sidepatch}), and (\ref{cornerpatch}) respectively.}
\label{fig:SquareRef}
\end{center}
\end{figure}
\par
 An explicit basis for the global approximation space $V_h$ is then given by
\begin{equation}\label{eq:GlobalBasis}
   \left\{\Psi_{j}\left[\hat{f}_{ji} \circ T_j^{-1} \right] :j=1,2,\dots,N; i=1,2,\dots ,9 \right\}.
\end{equation}
\begin{remark}\label{rem:NodePlacements}
  Since all the nodes in a rectangular patch $Q_j$ are located in $Q_j^{\text{flat}}$ where $\Psi_k=0$ for
  $k\neq j$, all the basis functions of $V_h$ vanish at the nodes in $Q_j$ except those
  associated with $Q_j$.
\end{remark}
\par
 Figure~\ref{fig:dofplacement} illustrates the degrees of freedom associated with
 the basis of the global approximation space for a square which is divided into $9$ square patches
 where $h_1=h_2=h$ and $\delta_1=\delta_2=\delta$.
\begin{figure}[h]
\begin{center}
 \includegraphics[scale=.9]{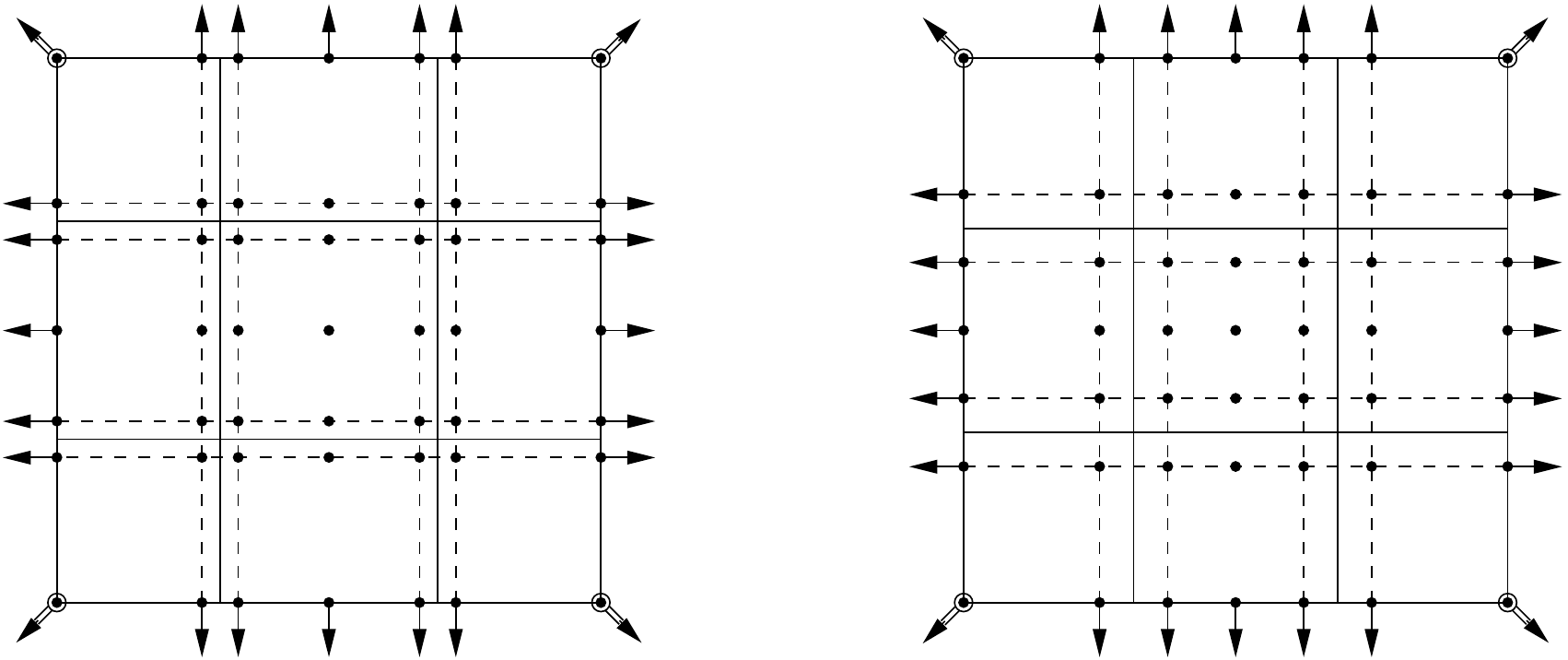}
 \caption{A partition of a domain $\Omega$ differing only by the choice of $\delta$.  The
  solid lines separate the different patches $Q_j, j=1, \dots, 9$.  The dashed lines represent
 the extension of $Q_j$ by $\delta(h/2)$ on each side.  This figure also shows the location of the
 degrees of freedom corresponding to $\delta=1/6$ (left) and $\delta=1/3$ (right).}
\label{fig:dofplacement}
\end{center}
\end{figure}
\begin{remark} \label{rem:LShaped}
  One may follow the same procedure for non-convex polygonal domains.  As an example,
 consider an  $L$-shaped domain $(-a,a)^2 \backslash
 [0,a]^2$.  One could divide the domain into rectangular patches
 everywhere except near the reentrant corner.  Near the reentrant corner, one could construct
 local biquadratic polynomial basis functions in the reference $L$-shaped domain $(-1,1)^2 \backslash
 [0,1]^2$ dual to the nodal variables
\goodbreak
%
\begin{alignat*}{9}
   &\phantom{a}&&&&&&&&\\[-40pt]
 N_1(v)&=v(0,0),&\quad& N_2(v)&&=v(1,0),&\quad& N_3(v)&&=v(0,1),   \\
 N_4(v)&=v(-1,-1),&\quad& N_5(v)&&=\frac{\partial v}{\partial \xi_1}(0,0),&\quad&
 N_6(v)&&=\frac{\partial v}{\partial \xi_1}(0,1),   \\
 N_7(v)&=\frac{\partial v}{\partial \xi_2}(0,0),&\quad& N_8(v)&&=\frac{\partial v}{\partial \xi_2}(0,1),
 &\quad& N_9(v)&&=\frac{\partial^2 v}{\partial \xi_1 \xi_2}(0,0),
\end{alignat*}
 as depicted in Figure~\ref{fig:LSHAPE}.
\begin{figure}[h]
\begin{center}
\includegraphics[scale=.9]{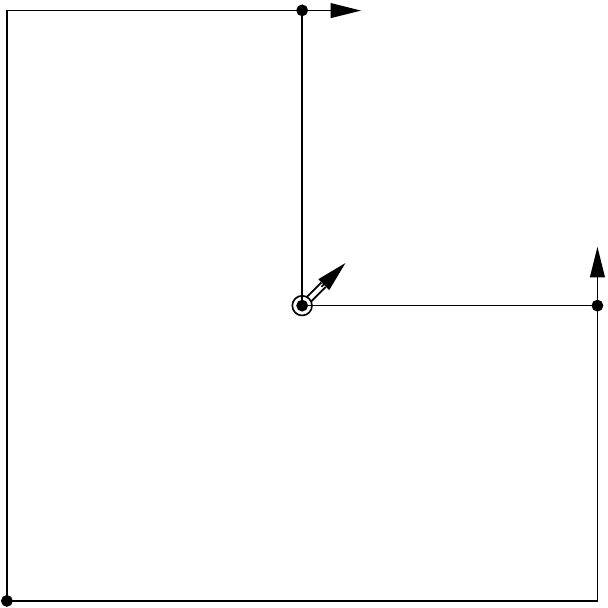}
\caption{Reference element for the $L$-shaped domain.}
\label{fig:LSHAPE}
\end{center}
\end{figure}
\end{remark}
\subsection{Interpolation Operator}\label{subsec:Interpolation}
 First we define interpolation operators associated with the rectangular patches.
 Let $\zeta\in H^2(\R^2)$.
\begin{itemize}\itemsep=5pt
\item For a patch with the local basis given by \eqref{innerpatch}
(cf. the reference element on the left of Figure~\ref{fig:SquareRef}),
 we define $\Pij\zeta$ to be the polynomial in $\bQ_2$ such that
   $(\Pij\zeta)\circ T_j=\zeta\circ T_j$
 at the 9 points in the set $\{(p,q):\,p,q=-1,0,1\}$.
\item For a patch with the local basis given by \eqref{sidepatch}
(cf. the reference element in the middle of Figure~\ref{fig:SquareRef}),
 we define $\Pij\zeta$ to be the polynomial in $\bQ_2$ such that
\smallskip
\begin{itemize}\itemsep=4pt
  \item[(i)] $(\Pij\zeta)\circ T_j=\zeta\circ T_j$ at the 6 points in the set
  $\{(p,q):p=-1,1$ and $q=-1,0,1\}$.
 \item[(ii)] The polynomial $({\p [(\Pij\zeta)\circ T_j]}/{\p \xi_1})|_{\xi_1=-1}$
 equals the quadratic polynomial
  $\hat\lambda(\xi_2)$, which is the $L_2$ projection of
  $({\p (\zeta\circ T_j)}/{\p \xi_1})|_{\xi_1=-1}$ into the space of quadratic
  polynomials in the variable $\xi_2$.
\end{itemize}
\item For a patch with the local basis given by \eqref{cornerpatch}
(cf. the reference element on the right of Figure~\ref{fig:SquareRef}),
 we define $\Pij\zeta$ to be the polynomial in $\bQ_2$ such that
\smallskip
\begin{itemize}\itemsep=4pt
\item[(i)] $(\Pij\zeta)\circ T_j=\zeta\circ T_j$ at the 4 points in the set
  $\{(p,q):p,q=\pm1\}$.
\item[(ii)] $(\p [(\Pij\zeta)\circ T_j]/\p \xi_1)|_{\xi_1=-1}$ equals the
     quadratic polynomial $\hat\lambda$ at $\xi_2=\pm1$,
   where $\hat\lambda(\xi_2)$ is the $L_2$ projection of
  $({\p (\zeta\circ T_j)}/{\p \xi_1})|_{\xi_1=-1}$ into the space of
    quadratic polynomials in the variable $\xi_2$.
\item[(iii)] $(\p [(\Pij\zeta)\circ T_j]/\p \xi_2)|_{\xi_2=-1}$ equals the
   quadratic polynomial $\hat\mu$ at $\xi_1=\pm1$,
   where $\hat\mu(\xi_1)$ is the $L_2$ projection of
  $({\p (\zeta\circ T_j)}/{\p \xi_2})|_{\xi_2=-1}$ into the space of
  quadratic polynomials in the variable $\xi_1$.
\item[(iv)] The value of
 $(\p^2[(\Pij\zeta)\circ T_j]/\p \xi_1\p \xi_2)$ at $(-1,-1)$ equals $(\hat\lambda'(-1)+\hat\mu'(-1))/2$.
\end{itemize}
\end{itemize}
\begin{remark}\label{rem:LocalInterpolationOperator}
  Since $T_j$ maps the reference square to $Q_j^{\text{flat}}$,
  the interpolant $\Pij\zeta$ is determined by the restriction of $\zeta$ to $Q_j^{\text{flat}}$.
\end{remark}
\par
 We can now define the global interpolation operator $\Pi_h:H^2(\O)\longrightarrow V_h$ by
\begin{equation*}
 \Pi_h \zeta =\sum_{j=1}^N (\Pi_j \zeta_\dag)\Psi_j\qquad\forall\,\zeta\in H^2(\O),
\end{equation*}
 where $\zeta_\dag\in H^2(\R^2)$ is any extension of $\zeta$.  The interpolant $\Pi_h$ is independent of the
 choice of $\zeta_\dag$ by Remark~\ref{rem:LocalInterpolationOperator}.  Moreover, by construction we have
\begin{equation}\label{eq:PihProperty}
  \Pi_h(H^2_0(\O))=V_h\cap H^2_0(\O).
\end{equation}
\par
 Let $\tQ_j$ be the rectangle centered at $y_j$ with width $h_1(1+\delta_1)=h_1+2\gamma_1$ and
 height $h_2(1+\delta_2)=h_2+2\gamma_2$.  Let $h=\max(h_1,h_2)$.  Since $\Pij P=P$ for any $P\in\bQ_2$, the estimate
\begin{equation}\label{eq:LocalInterpolationError}
   \sum_{m=0}^2h^m|\zeta-\Pij \zeta|_{H^m(\tQ_j\cap\O)}\leq Ch^{2+\alpha}|\zeta|_{H^{2+\alpha}(\tQ_j\cap\O)}
\end{equation}
 follows from the Bramble-Hilbert lemma  \cite{BH:1970:Lemma,DS:1980:BH} and scaling.  From here on we use $C$
 to denote a generic positive constant that is independent of the mesh size $h$.
\par
 Combining the local interpolation error estimate \eqref{eq:LocalInterpolationError} and the estimates
  for the partition of unity functions $\Psi_j$ in \cite{OKH:2008:PTU}, we immediately have
  (cf. \cite{MB:1996:PTU,OKH:2008:PTU}) the following error estimates for the global
  interpolation operator $\Pi_h$:
\begin{equation}\label{eq:PihEst}
  \sum_{m=0}^2 h^m|\zeta-\Pi_h \zeta|_{H^m(\O)} \leq C h^{2+\alpha}|\zeta |_{H^{2+\alpha}(\O)}.
\end{equation}
\begin{remark}\label{rem:Bicubic}
  Classical rectangular $C^1$ finite element methods
  would require a local approximation space that is at least bi-cubic \cite{BFS:1965:Element}.
  Of course we can also use bi-cubic polynomials as the local approximation space in
  our GFEM (cf. \cite{Davis:2011:Thesis,ODJ:2011:Plate} and Example~1 in Section~\ref{sec:Numerics}).
\end{remark}
%
\subsection{The Discrete Obstacle Problem}\label{subsec:DiscreteObstacle}
 Let $\cV_h$ be the set of the nodes in the rectangular patches corresponding to the degrees of
 freedom involving pointwise evaluation of the local basis functions.  (Such nodes
 are represented by solid dots in Figure~\ref{fig:dofplacement} and Figure~\ref{fig:LSHAPE}.)
 The GFEM for the model problem is to find $u_h\in K_{h}$ such that
\begin{equation}\label{eq:DiscreteObstacle}
 u_h=\mathop{\rm argmin}_{v\in K_h} G(v),
\end{equation}
 where the quadratic functional $G$ is defined by \eqref{eq:GDef}--\eqref{eq:aDef} and
\begin{eqnarray}\label{eq:KhDef}
 K_h &=& \{v\in V_h : v-\Pi_h g\in H^2_0(\Omega),\, \psi_1(p) \leq v(p) \leq \psi_2(p)
\quad \forall p\in \mathcal{V}_h \}.
\end{eqnarray}
\begin{remark}\label{rem:Kh} Approximation of the essential boundary conditions $u=g$ and
 ${\partial u}/{\partial n}={\partial g}/{\partial n}$ are both
 included in the definition of $K_{h}$.  Moreover $K_{h}$ is nonempty because $\Pi_h K\subset K_{h}$ by
 \eqref{eq:KDef} and \eqref{eq:PihProperty}.
\end{remark}
\begin{remark}\label{rem:BoxConstraints}
  In view of Remark~\ref{rem:NodePlacements} and the defining properties of the polynomials
  $L_i$ and $H_i$, the constraints defining $K_h$ are box constraints with respect to the
  basis of $V_h$ defined in \eqref{eq:GlobalBasis}.
\end{remark}
\par
 It follows from the standard theory that the discrete obstacle problem (\ref{eq:DiscreteObstacle})
 has a unique solution characterized by the discrete variational inequality
\begin{equation}\label{eq:DiscreteVI}
 a(u_h, v-u_h)\geq (f, v-u_h) \quad \forall v \in K_h.
\end{equation}
%
\section{Convergence Analysis}\label{sec:Convergence}
%
 We begin with some preliminary estimates in Section~\ref{subsec:Preliminary} and introduce an
 auxiliary obstacle problem in Section~\ref{subsec:AuxiliaryProblem} that connects the continuous problem
 \eqref{eq:ObstacleProblem} and the discrete problem \eqref{eq:DiscreteObstacle}.  The main result is derived
 in Section~\ref{subsec:Errors}.
\subsection{Preliminary Estimates}\label{subsec:Preliminary}
 In view of \eqref{eq:PihEst}, it suffices to find an optimal estimate for $|\Pi_hu-u_h|_{H^2(\O)}$.
 Using the discrete variational inequality \eqref{eq:DiscreteVI}, we have
\begin{align*}
 |\Pi_hu-u_h|_{H^2(\O)}^2&=a(\Pi_hu-u,\Pi_hu-u_h)+a(u-u_h,\Pi_hu-u_h)\\
      &\leq |\Pi_hu-u|_{H^2(\O)}|\Pi_hu-u_h|_{H^2(\O)}+a(u,\Pi_hu-u_h)-(f,\Pi_hu-u_h)\\
      &\leq \frac12|\Pi_hu-u|_{H^2(\O)}^2+\frac12|\Pi_hu-u_h|_{H^2(\O)}^2
      +a(u,\Pi_hu-u_h)-(f,\Pi_hu-u_h),
\end{align*}
 which implies
\begin{equation}\label{eq:PreliminaryEst}
  |\Pi_hu-u_h|_{H^2(\O)}^2\leq |\Pi_hu-u|_{H^2(\O)}^2+2\big[a(u,\Pi_hu-u_h)-(f,\Pi_hu-u_h)\big].
\end{equation}
 We can therefore complete the error analysis by finding an optimal estimate for the expression
 $a(u,\Pi_hu-u_h)-(f,\Pi_hu-u_h)$.
\par
 The following result is useful for the error analysis in Section~\ref{subsec:Errors}.
\begin{lemma}\label{lem:2alpha}
  There exists a positive constant $C$ independent of $h$ such that
\begin{equation}\label{eq:2alphaEst}
  |a(\phi,\zeta-\Pi_h\zeta)|\leq Ch^{2\alpha}\|\phi\|_{H^{2+\alpha}(\O)}
    \|\zeta\|_{H^{2+\alpha}(\O)}
\end{equation}
  for all
   $ \phi\in H^{2+\alpha}(\O)$ and $\zeta\in H^{2+\alpha}(\O)\cap H^2_0(\O)$.
\end{lemma}
\begin{proof}  Let $\zeta\in H^{2+\alpha}(\O)\cap H^2_0(\O)$ be arbitrary.
  On the one hand we  have an obvious estimate
\begin{align}\label{eq:2alphaEst1}
 |a(\phi,\zeta-\Pi_h\zeta)|&\leq |\phi|_{H^2(\O)}|\zeta-\Pi_h\zeta|_{H^2(\O)}\\
   &\leq Ch^\alpha|\phi|_{H^2(\O)}|\zeta|_{H^{2+\alpha}(\O)}
   \qquad\forall\,\phi\in H^2(\O)\notag
\end{align}
 that follows from \eqref{eq:PihEst}.  On the other hand,
 we have another estimate
\begin{align}\label{eq:2alphaEst2}
 |a(\phi,\zeta-\Pi_h\zeta)|&=\Big|\int_\O (\nabla\cdot\nabla^2 \phi)\cdot\nabla(\zeta-\Pi_h\zeta)\,dx\Big|
   \notag\\
   &\leq C|\phi|_{H^3(\O)}|\zeta-\Pi_h\zeta|_{H^1(\O)}\\
   &\leq Ch^{1+\alpha}|\phi|_{H^3(\O)}|\zeta|_{H^{2+\alpha}(\O)}
   \qquad\forall \,\phi\in H^3(\O)\notag
\end{align}
  that follows from \eqref{eq:PihProperty},
  \eqref{eq:PihEst} and  integration by parts.
\par
  The estimate \eqref{eq:2alphaEst}  follows from \eqref{eq:2alphaEst1},
  \eqref{eq:2alphaEst2} and interpolation between Sobolev spaces
  \cite{ADAMS:2003:Sobolev,Tartar:2007:Sobolev}.
\end{proof}
\subsection{An Auxiliary Obstacle Problem}\label{subsec:AuxiliaryProblem}
 We can connect the continuous obstacle problem \eqref{eq:ObstacleProblem} and
 the discrete obstacle problem (\ref{eq:DiscreteObstacle}) through an intermediate obstacle problem:
 Find $\tilde{u}_h \in \tilde{K}_h$ such that
\begin{equation}\label{auxiliary_prob}
\tilde{u}_h=\mathop{\rm argmin}_{v\in \tilde K_h} G(v)
\end{equation}
 where
\begin{equation}\label{eq:tKhDef}
  \tilde{K}_h = \{ v\in H^2(\Omega) : v-g \in H^2_0(\Omega), \psi_1(p) \leq v(p)
  \leq \psi_2(p) \quad \forall p \in \mathcal{V}_h \}.
\end{equation}
\par
 Note that $\tilde{K}_h$ is a closed convex subset of $H^2(\Omega)$ and $K \subset \tilde{K}_h.$
 The unique solution of (\ref{auxiliary_prob}) is characterized by the variational inequality:
\begin{equation}\label{eq:AuxiliaryVI}
 a(\tilde{u}_h, v-\tilde{u}_h) \geq (f, v-\tilde{u}_h) \quad \forall v \in \tilde{K}_h.
\end{equation}
\par
 The connection between \eqref{eq:ObstacleProblem} and (\ref{auxiliary_prob})
 is given by the following properties of $\tilde{u}_h$ from \cite{BSZ:2012:Obstacle,BSZZ:2012:Kirchhoff}:
\begin{equation}\label{eq:tuhEst}
 |u-\tilde{u}_h|_{H^2(\Omega)}\leq C h,
\end{equation}
 and there exists $h_0 >0$ such that
\begin{equation}\label{eq:uhat}
 \hat{u}_h=\tilde{u}_h+\delta_{h,1}\phi_1 - \delta_{h,2} \phi_2 \in K \quad \forall \,h \leq h_0,
\end{equation}
 where $\phi_1$ and $\phi_2$ are $C^\infty$ functions with compact supports in $\O$ such that
 $\phi_i=1$ on the coincidence set $\{x\in\O:\,u(x)=\psi_i(x)\}$,
 and the positive numbers $\delta_{h,1}$ and
 $\delta_{h,2}$ satisfy
\begin{equation}\label{eq:deltah}
\delta_{h,i} \leq C h^2.
\end{equation}
\par
 Note that $v-\Pi_hg\in H^2_0(\O)$ for all $v\in K_h$ (cf. \eqref{eq:KhDef}) and hence, by
  \eqref{eq:tKhDef},
\begin{equation}\label{eq:KhandtKh}
  v+(g-\Pi_h g)\in \tK_h\qquad\forall\,v\in K_h.
\end{equation}
\subsection{Error Estimates for the Generalized Finite Element Method}\label{subsec:Errors}
 We now complete the error analysis of the generalized finite
 element method by deriving an optimal estimate for the expression $a(u,\Pi_h-u_h)-(f,\Pi_hu-u_h)$.
 To simplify the presentation, we introduce the transitive relation $A\dotleq B$ defined by
\begin{equation*}
  A\dotleq B\Leftrightarrow A-B\leq C(h^{2\alpha}+h^\alpha |\Pi_hu-u_h|_{H^2(\O)}).
\end{equation*}
\par
 Since
\begin{equation*}
  a(u,\Pi_hu-u_h)=a(u-\tu_h,\Pi_hu-u_h)+a(\tu_h,\Pi_hu-u_h)
\end{equation*}
 and
\begin{equation*}
  a(u-\tu_h,\Pi_hu-u_h)\leq |u-\tu_h|_{H^2(\O)}|\Pi_hu-u_h|_{H^2(\O)}\leq Ch
   |\Pi_hu-u_h|_{H^2(\O)}
\end{equation*}
 by the estimate \eqref{eq:tuhEst},  we have
\begin{equation}\label{eq:DotInequality1}
 a(u,\Pi_hu-u_h)-(f,\Pi_hu-u_h)\dotleq a(\tu_h,\Pi_hu-u_h)-(f,\Pi_hu-u_h).
\end{equation}
\par
 In view of \eqref{eq:KhandtKh}, we can use the auxiliary variational inequality
 \eqref{eq:AuxiliaryVI} to obtain
\begin{align*}
   a(\tu_h,\Pi_hu-u_h)&=a(\tu_h,\tu_h-u_h-(g-\Pi_hg))+a(\tu_h,\Pi_hu-\tu_h+(g-\Pi_hg))\\
         &\leq (f,\tu_h-u_h-(g-\Pi_hg))+a(\tu_h,\Pi_hu-\tu_h+(g-\Pi_hg)),
\end{align*}
 which together with \eqref{eq:DotInequality1} implies
\begin{align}\label{eq:DotInequality2}
  &a(u,\Pi_hu-u_h)-(f,\Pi_hu-u_h)\dotleq \\
    &\hspace{50pt}a(\tu_h,\Pi_hu-\tu_h+(g-\Pi_hg))-(f,u-\tu_h)-(f,(\Pi_hu-u)+(g-\Pi_hg)).\notag
\end{align}
\par
 We can rewrite the first term on the right-hand side of \eqref{eq:DotInequality2} as
\begin{align*}
  a(\tu_h,\Pi_hu-\tu_h+(g-\Pi_hg))&=a(\tu_h-u,\Pi_hu-\tu_h+(g-\Pi_hg))+a(u,\Pi_h(u-g)-(u-g))\\
       &\hspace{30pt}+a(u,u-\tu_h).
\end{align*}
 Observe that
\begin{align*}
  &a(\tu_h-u,\Pi_hu-\tu_h+(g-\Pi_hg))=a(\tu_h-u,(\Pi_hu-u)+(u-\tu_h)+(g-\Pi_hg))\\
      &\hspace{70pt}\leq |\tu_h-u|_{H^2(\O)}\big(|\Pi_hu-u|_{H^2(\O)}
      +|u-\tu_h|_{H^2(\O)}+|g-\Pi_hg|_{H^2(\O)}\big)\\
      &\hspace{70pt}\leq Ch^{1+\alpha}
\end{align*}
 by  \eqref{eq:PihEst} and  \eqref{eq:tuhEst}, and
\begin{equation*}
   a(u,\Pi_h(u-g)-(u-g))  \leq Ch^{2\alpha}
\end{equation*}
 by Lemma~\ref{lem:2alpha}.  Moreover we have, by \eqref{eq:PihEst},
\begin{equation*}
  -(f,(\Pi_hu-u)+(g-\Pi_hg))\leq \|f\|_{L_2(\O)}\big(\|\Pi_hu-u\|_{L_2(\O)}+\|g-\Pi_hg\|_{L_2(\O)}\big)
     \leq Ch^{2+\alpha}.
\end{equation*}
\par
 Combining these relations and \eqref{eq:DotInequality2}, we
 arrive at the estimate
\begin{equation}\label{eq:DotInequality3}
  a(u,\Pi_hu-u_h)-(f,\Pi_hu-u_h).
    \dotleq a(u,u-\tu_h)-(f,u-\tu_h).
\end{equation}
\par
 According to \eqref{eq:VariationalInequality}, \eqref{eq:uhat} and \eqref{eq:deltah},  we have
\begin{align*}
 & a(u,u-\tu_h)-(f,u-\tu_h)\\
   &\hspace{30pt}=a(u,u-\hat u_h)-(f,u-\hat u_h)+\delta_{h,1}[a(u,\phi_1)-(f,\phi_1)]
    -\delta_{h,2}[a(u,\phi_2)-(f,\phi_2)]\\
           &\hspace{30pt}\leq  Ch^2,
\end{align*}
\goodbreak\noindent
 and hence \eqref{eq:DotInequality3} leads to the estimate
\begin{equation*}
  a(u,\Pi_hu-u_h)-(f,\Pi_hu-u_h)
    \dotleq 0,
\end{equation*}
 which means
\begin{equation}\label{eq:KeyEst}
  a(u,\Pi_hu-u_h)-(f,\Pi_hu-u_h)\leq C(h^{2\alpha}+h^\alpha|\Pi_hu-u_h|_{H^2(\O)}).
\end{equation}
\begin{theorem}\label{thm:Error}
  There exists a positive constant $C$ independent of $h$ such that
\begin{equation*}
 |u-u_h|_{H^2(\O)}\leq Ch^{\alpha}.
\end{equation*}
\end{theorem}
\begin{proof}  It follows from \eqref{eq:PihEst}, \eqref{eq:PreliminaryEst}, \eqref{eq:KeyEst}  and
 the arithmetic and geometric means inequality that
\begin{equation*}
 |\Pi_hu-u_h|_{H^2(\O)}^2\leq C(h^{2\alpha}+h^\alpha|\Pi_hu-u_h|_{H^2(\O)})
    \leq Ch^{2\alpha}+\frac12|\Pi_hu-u_h|_{H^2(\O)}^2,
\end{equation*}
 which implies
\begin{equation}\label{eq:Pihuanduh}
 |\Pi_hu-u_h|_{H^2(\O)}\leq Ch^\alpha.
\end{equation}
\par
 The theorem  follows from \eqref{eq:PihEst}, \eqref{eq:Pihuanduh} and the
 triangle inequality.
\end{proof}
\par
 Since $H^2(\O)$ is embedded in $C(\bar\O)$ by the Sobolev embedding
 theorem \cite{ADAMS:2003:Sobolev,Tartar:2007:Sobolev}, the following corollary is
 immediate.  But numerical results in Section~\ref{sec:Numerics} indicate that the convergence rate in the
 $L_\infty(\O)$ norm should be higher than the convergence rate in the $H^2(\O)$ norm.
\begin{corollary}\label{cor:LInftyError}
  There exists a positive constant $C$ independent of $h$ such that
\begin{equation}\label{eq:LInftyEst}
 |u-u_h|_{L_\infty(\O)}\leq Ch^{\alpha}.
\end{equation}
\end{corollary}
\begin{remark}\label{rem:CoincidenceSet}
  Under additional assumptions \cite{BC:1983:FreeBoundary,Nochetto:1986:FreeBoundary}
  on the exact coincidence sets (resp. free boundaries), the error
  estimate \eqref{eq:LInftyEst} implies the convergence of the discrete coincidence sets (resp.
  free boundaries) to the exact coincidence sets (resp. free boundaries).  Details can be
  found in \cite{BSZZ:2012:Kirchhoff}.
\end{remark}
\section{Numerical Results}\label{sec:Numerics}
 We present numerical results for several one-obstacle problems to
 demonstrate the performance of the GFEM.  The obstacle function from below will be
 denoted by $\psi.$  The first four examples are
 from \cite{BSZZ:2012:Kirchhoff}.  The discrete obstacle problems are solved by
 a primal dual active set strategy from \cite{BIK:1999:PDAS,HIK:2002:PDAS}.

\par\smallskip
{\bf Example 1.}
 Here we apply the GFEM to a problem with a known exact solution to
 validate the numerical results.  We begin with the plate obstacle problem on the disc
 $\{x:|x|<2\}$ with $f=0, \psi(x)=1-|x|^2$ and homogeneous Dirichlet boundary conditions.
 This problem is rotationally invariant and
 can be solved exactly.  The exact solution is
\begin{equation*}
 u(x) =
\begin{cases}
 C_1|x|^2 \ln|x|+C_2|x|^2+C_3\ln|x|+C_4, & r_0<|x|<2 \\
 1-|x|^2, & |x| \leq r_0
\end{cases}
\end{equation*}
 where $r_0\approx 0.18134452, C_1 \approx 0.52504063, C_2 \approx -0.62860904,
 C_3 \approx 0.01726640,$ and $C_4 \approx 1.04674630.$
 We then consider the obstacle problem on $\O = (-0.5,0.5)^2$ whose exact solution is
 the restriction of $u$ to $\O.$  For this problem $f=0, \psi(x)=1-|x|^2$ and the (non-homogeneous)
 Dirichlet boundary data are determined by $u.$
\par
 We partition $\O$ following the procedure described in Section~\ref{subsec:ApproximationSpace} and define
 $j$ to be the level where there are $2^{j}$ equal subdivisions in each direction.
 We solve the discrete obstacle problem on each level $j$ with $\delta=1/3$ so that
 the mesh parameter $h_j=(2^j-1/3)^{-1}$.
\par
 We denote the energy norm on the $j$-th level by $\| \cdot \|_j$.
 Let $u_j$ be the numerical solution of the $j$-th level discrete obstacle problem and
 $e_j=\Pi_j u-u_j$ where $\Pi_j$ is the interpolation operator on the $j-$th level.
 We evaluate the error $\|e_j\|_j$ in the energy norm, and the error $\|e_j\|_{\infty}$
 in the $\ell_{\infty}$ norm,
  and compute the rates of convergence in these norms by
\begin{equation*}
\beta_h=\ln(\|e_{j/2}\|_{j/2}/\|e_j\|_j)/\ln(h_{j/2}/h_j)\quad \text{ and } \quad
 \beta_{\infty}=\ln(\|e_{j/2}\|_{\infty}/\|e_j\|_{\infty})/\ln(h_{j/2}/h_j).
\end{equation*}
\par
 The numerical results are presented in Table~\ref{Ex1}. It is observed that the magnitude
 of the error in energy norm is $O(h)$.
\begin{table}[hhh]
\begin{center}
 \begin{tabular}{ | c | c r | c r | }
 \hline
  &&&&\\[-12pt]
 $j$ & $\|e_j\|_j/\|u_{8}\|_{8} $ & $\beta_h$ & $\|e_j\|_{\infty} $ & $\beta_{\infty}$ \\
 \hline &&&&\\[-12pt]
  1 &   0.0000  $\times 10^{-0}$ &    &   0.0000  $\times10^{-0}$ &   \\
 \hline
  &&&&\\[-12pt]
  2 &   1.2365  $\times 10^{-1}$ &     &   8.8312  $\times 10^{-4}$ &     \\
 \hline &&&&\\[-12pt]
  3 &   6.3226  $\times 10^{-2}$ &   0.9094 &   6.0088  $\times 10^{-4}$ &   0.5221 \\
 \hline &&&&\\[-12pt]
  4 &   2.5977  $\times 10^{-2}$ &   1.2447 &   8.8401  $\times 10^{-5}$ &   2.6817 \\
 \hline &&&&\\[-12pt]
  5 &   1.2159 $\times 10^{-2}$ &   1.0787 &   2.4443  $\times 10^{-5}$ &   1.8267 \\
  \hline&&&&\\[-12pt]
  6 &   5.9045 $\times 10^{-3}$ &   1.0343 &   6.7946 $\times 10^{-6}$ &   1.8331 \\
 \hline &&&&\\[-12pt]
  7 &   2.9125 $\times 10^{-3}$ &   1.0157 &   1.4775  $\times10^{-6}$ &   2.1929 \\
 \hline &&&&\\[-12pt]
  8 &   1.4396 $\times 10^{-3}$ &   1.0147 &   8.8608 $\times10^{-7}$ &   0.7363 \\\hline
   \end{tabular}
 \end{center}
\par\medskip
 \caption{Energy norm and $\ell_{\infty}$ norm errors for Example 1.}\label{Ex1}
 \end{table}
\par
 The exact coincidence set $I$ for this example is the disc centered at $(0,0)$ with radius $r_0$.
 Let $\mathcal{V}_j$ be the set of nodes on the $j$-th level corresponding to
 degrees of freedom involving pointwise evaluation of local basis functions in the
 interior of $\Omega$.  Then we define the discrete coincidence set $I_j$ by
\begin{equation*}
 I_j=\{p\in \mathcal{V}_j:u_j(p)-\psi(p)\leq \|e_j\|_{\infty}\}.
\end{equation*}
 The discrete coincidence sets $I_7$ and $I_8$
 are displayed in Figure~\ref{fig:CoincidenceSetExample1Q2}, where the radius of
 the circle in black is $r_0$.  The convergence of the discrete coincidence sets is observed.
\begin{figure}[h]
\centering\includegraphics[scale=0.5]{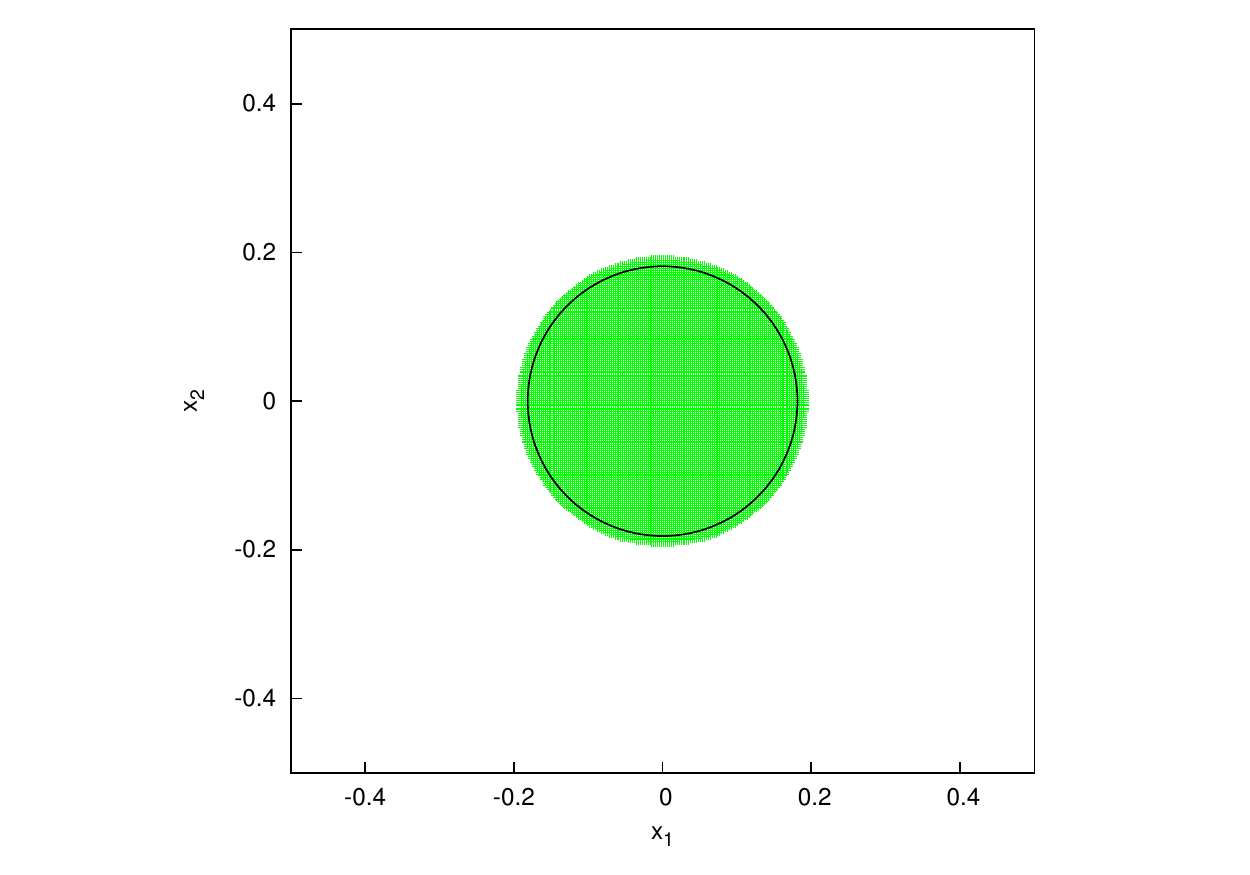}
\centering\includegraphics[scale=0.5]{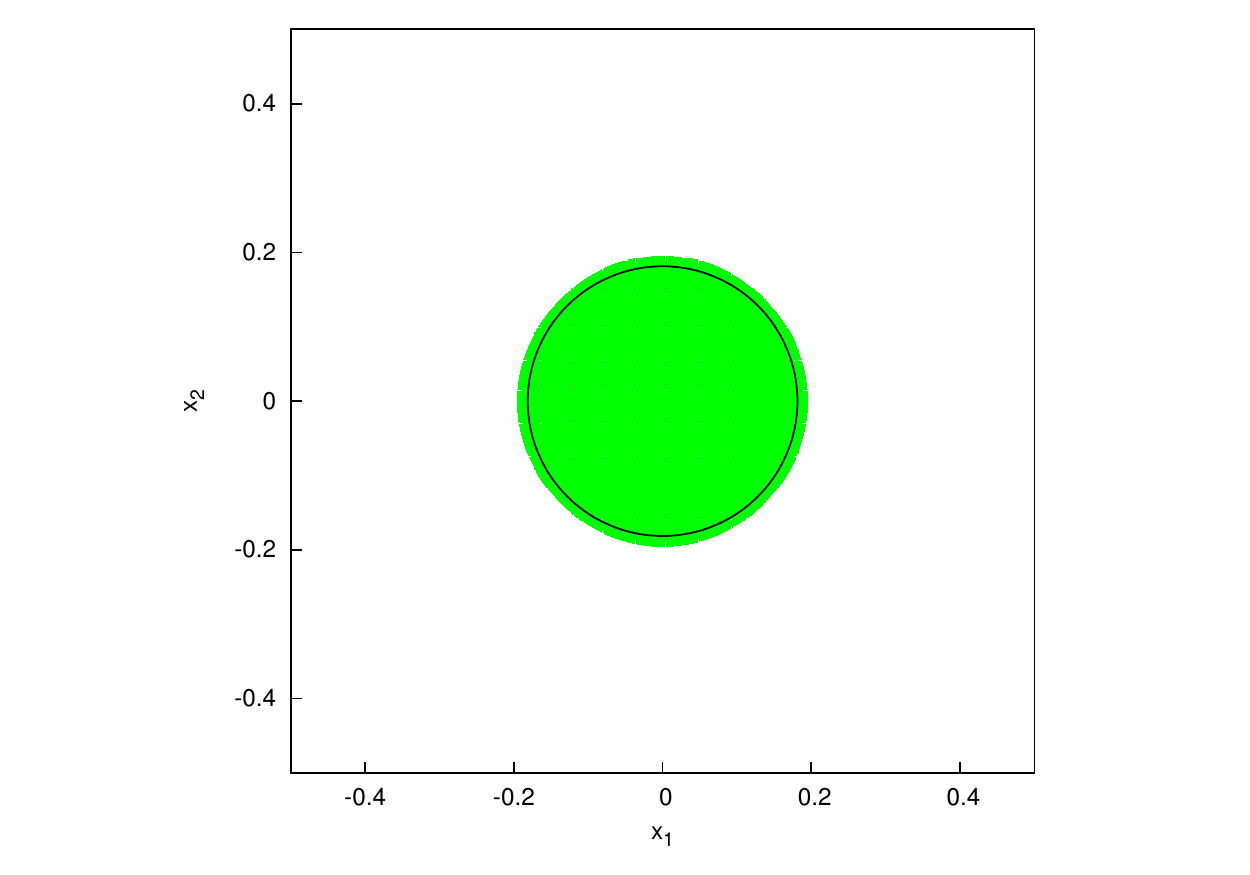}
\caption{Discrete coincidence set for Example 1 for level 7 (left) and level 8
(right).}
\label{fig:CoincidenceSetExample1Q2}
\end{figure}
\par
 One of the advantages of the GFEM is that the local approximation space can be easily adjusted.
 In Table~\ref{Ex1b} we report the numerical results for the same problem but with
  $\mathbb{Q}_3$ as the local approximation space.  An $O(h^{1.5})$ energy error is
  observed, which is due to the fact that the exact solution $u$ is piecewise smooth.
\par
 \begin{table}[h]
 \begin{center}
 \begin{tabular}{ | c | c r | c r | }
 \hline
 &&&&\\[-12pt]
 $j$ & $\|e_j\|_j/\|u_{8}\|_{8} $ & $\beta_h$ & $\|e_j\|_{\infty} $ & $\beta_{\infty}$ \\
 \hline&&&&\\[-12pt]
  1 &   1.4199   $\times 10^{-2}$ &    &   1.0561   $\times 10^{-4}$ &   \\ \hline&&&&\\[-12pt]
  2 &   6.1489   $\times 10^{-2}$ &  -1.8589   &   7.6281   $\times 10^{-4}$ &   -2.5078  \\ \hline&&&&\\[-12pt]
  3 &   1.8374   $\times 10^{-2}$ &   1.6377 &   9.4507  $\times 10^{-5}$ &   2.8313 \\ \hline&&&&\\[-12pt]
  4 &   5.8004  $\times 10^{-3}$ &   1.6134 &   1.4396  $\times 10^{-5}$ &   2.6330 \\ \hline&&&&\\[-12pt]
  5 &   2.3728   $\times 10^{-3}$ &   1.2702 &   4.7114  $\times 10^{-6}$ &   1.5872 \\ \hline&&&&\\[-12pt]
  6 &   8.3768  $\times 10^{-4}$ &   1.4908 &   4.1685  $\times 10^{-7}$ &   3.4723 \\ \hline&&&&\\[-12pt]
  7 &   2.7675  $\times 10^{-4}$ &   1.5918 &   3.7129  $\times 10^{-6}$ &   -3.1431 \\ \hline
   \end{tabular}
 \end{center}
\par\medskip
\caption{Energy norm and $\ell_{\infty}$ norm errors for Example 1 with $\mathbb{Q}_3$ as local approximation space.}
\label{Ex1b}
 \end{table}
\begin{remark}\label{rem:ellInfty}
 Note that the $\ell_\infty$ norm errors fluctuate.  This is likely due to
 the fact that the primal dual active set strategy is based on
 stopping conditions that are unrelated to the $\ell_\infty$ norm.
\end{remark}
\par\smallskip\goodbreak
 {\bf Example 2.}
 In this example we take $\O=(-0.5,0.5)^2, f=g=0$ and $\psi(x)=1-5|x|^2+|x|^4.$
 We solve the discrete obstacle problems using the same PU functions as in Example 1.
\par
 Since the exact solution is not known, we take $\tilde{e}_j=\Pi_j u_{j-1}-u_j$ and
 compute the rates of convergence $\tilde{\beta}_h$ and $\tilde{\beta}_{\infty}$ by
\begin{equation*}
\tilde{\beta}_h=\ln(\|\tilde{e}_{j/2}\|_{j/2}/\|\tilde{e}_j\|_j)/\ln(h_{j/2}/h_j) \quad\text{ and } \quad
 \tilde{\beta}_{\infty}=\ln(\|\tilde{e}_{j/2}\|_{\infty}/\|\tilde{e}_j\|_{\infty})/\ln(h_{j/2}/h_j).
\end{equation*}
 The results are presented in Table~\ref{Ex2}.
\begin{table}[h]
 \begin{center}
 \begin{tabular}{ | c | c r | c r | }
 \hline
 &&&&\\[-12pt]
 $j$ & $\|\tilde{e}_j\|_j/\|u_{8}\|_{8} $ & $\tilde{\beta}_h$ & $\|\tilde{e}_j\|_{\infty} $ &
 $\tilde{\beta}_{\infty}$ \\ \hline
 &&&&\\[-12pt]
  1 &   2.9288   $\times10^{-0}$ &    &   9.0040  x $\times10^{-1}$ &   \\ \hline &&&&\\[-12pt]
  2 &   5.9820   $\times10^{-0}$ &   -0.9058  &   5.3416  x $\times10^{-1}$ & 0.6622    \\ \hline&&&&\\[-12pt]
  3 &   1.2402   $\times10^{-0}$ &   2.1333 &   5.2357   $\times10^{-1}$ &   0.0271 \\ \hline&&&&\\[-12pt]
  4 &   6.5242   $\times10^{-1}$ &   0.8988 &   2.5914   $\times10^{-2}$ &   4.2061 \\ \hline&&&&\\[-12pt]
  5 &   1.8496   $\times10^{-1}$ &   1.7913 &   1.7757   $\times10^{-3}$ &   3.8091 \\ \hline&&&&\\[-12pt]
  6 &   8.9273   $\times10^{-2}$ &   1.0430 &   4.4337   $\times10^{-4}$ &   1.9867 \\ \hline&&&&\\[-12pt]
  7 &   4.4296   $\times10^{-2}$ &   1.0072 &   1.1284   $\times10^{-4}$ &   1.9667 \\ \hline&&&&\\[-12pt]
  8 &   2.2154   $\times10^{-2}$ &   0.9977 &   3.7776   $\times10^{-5}$ &   1.5758 \\ \hline
   \end{tabular}
 \end{center}
\par\medskip
\caption{Energy norm and $\ell_{\infty}$ norm errors for Example 2.}
\label{Ex2}
 \end{table}
\par
 Since $\Delta^2\psi-f>0$ in this example, the non-coincidence set is known to be connected
 \cite{CF:1979:BiharmonicObstacle}.  This is confirmed by
 the discrete coincidence sets $I_7$ and $I_8$ displayed in Figure~\ref{fig:CoincidenceSetExample2}.
 Note that the discrete coincidence sets have the correct symmetries: rotations by right angles and
 reflections across coordinates axes.
\begin{figure}[h]
\centering\includegraphics[scale=0.5]{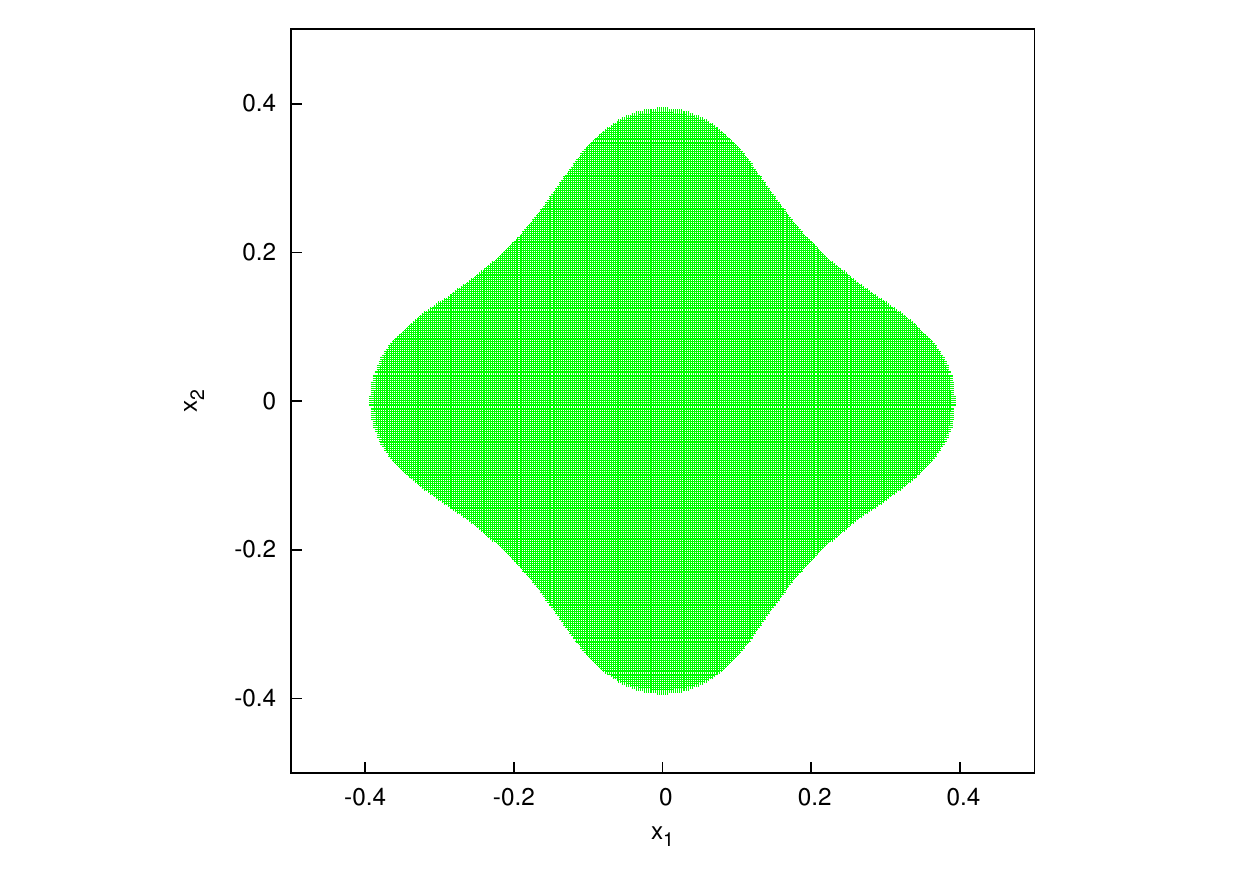}
\centering\includegraphics[scale=0.5]{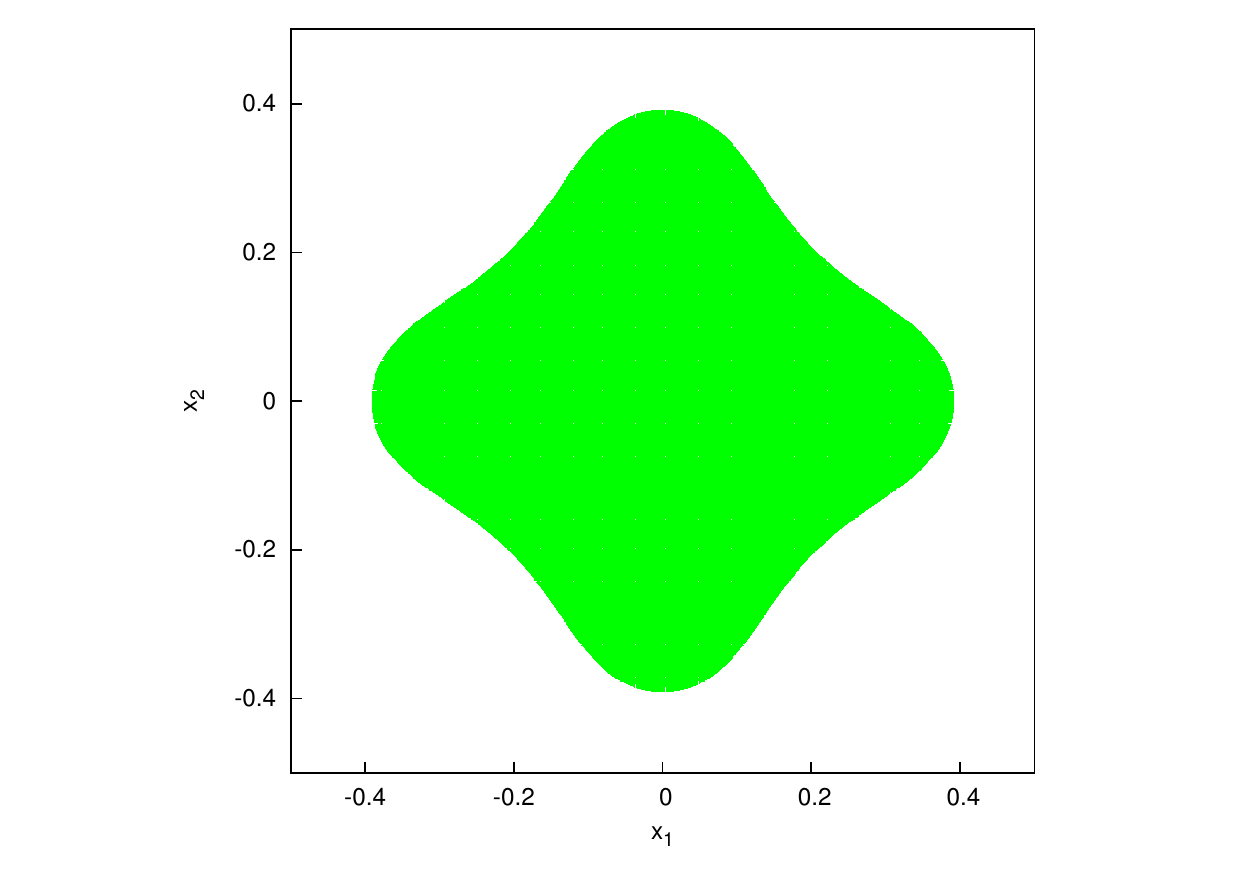}
\caption{Discrete coincidence set for Example 2 for level 7 (left) and level 8 (right).}
\label{fig:CoincidenceSetExample2}
\end{figure}
\par\smallskip\goodbreak
{\bf Example 3.}
 In this example we take $\O=(-0.5,0.5)^2$, $f=g=0$ and $\psi(x)=1-5|x|^2-|x|^4$.
 We solve the discrete obstacle problems using the same
 PU functions as in Example 1.  Numerical results are tabulated in Table~\ref{Ex3}.
\begin{table}[hh]
 \begin{center}
 \begin{tabular}{ | c | c r | c r | }
 \hline &&&&\\[-12pt]
 $j$ & $\|\tilde{e}_j\|_j/\|u_{8}\|_{8} $ & $\tilde{\beta}_h$
 & $\|\tilde{e}_j\|_{\infty} $ & $\tilde{\beta}_{\infty}$ \\ \hline &&&&\\[-12pt]
  1 &   3.0796   $\times10^{-0}$ &    &   8.9960   $\times10^{-1}$ &   \\ \hline &&&&\\[-12pt]
  2 &   6.2833   $\times10^{-0}$ &   -0.9044  &   4.8507   $\times10^{-1}$ & 0.7834    \\ \hline &&&&\\[-12pt]
  3 &   1.0279   $\times10^{-0}$ &   2.4544 &   4.3181   $\times10^{-1}$ &   0.1576 \\ \hline &&&&\\[-12pt]
  4 &   2.9125   $\times10^{-1}$ &   1.7646 &   1.9025   $\times10^{-2}$ &   4.3689 \\ \hline&&&&\\[-12pt]
  5 &   1.4890   $\times10^{-1}$ &   0.9533 &   1.6296   $\times10^{-3}$ &   3.4920 \\ \hline&&&&\\[-12pt]
  6 &   7.1583   $\times10^{-2}$ &   1.0487 &   4.8682   $\times10^{-4}$ &   1.7300 \\ \hline&&&&\\[-12pt]
  7 &   3.6108   $\times10^{-2}$ &   0.9836 &   1.3055   $\times10^{-4}$ &   1.8916 \\ \hline&&&&\\[-12pt]
  8 &   1.8072   $\times10^{-2}$ &   0.9966 &   3.1174   $\times10^{-5}$ &   2.0623 \\ \hline
   \end{tabular}
 \end{center}
\par\medskip
\caption{Energy norm and $\ell_{\infty}$ norm errors for Example 3.}\label{Ex3}
 \end{table}
\par
 The set-up for Example~3 is very similar to that of Example~2, except that now
 $\Delta^2\psi-f<0$ and hence the interior of the coincidence set must be empty, otherwise the
 complementarity form of the variational inequality would be violated.  This is
 confirmed by the discrete coincidence sets in Figure~\ref{fig:CoincidenceSetExample3}, which
 also possess  the correct symmetries.
\begin{figure}[h]
\centering\includegraphics[scale=0.5]{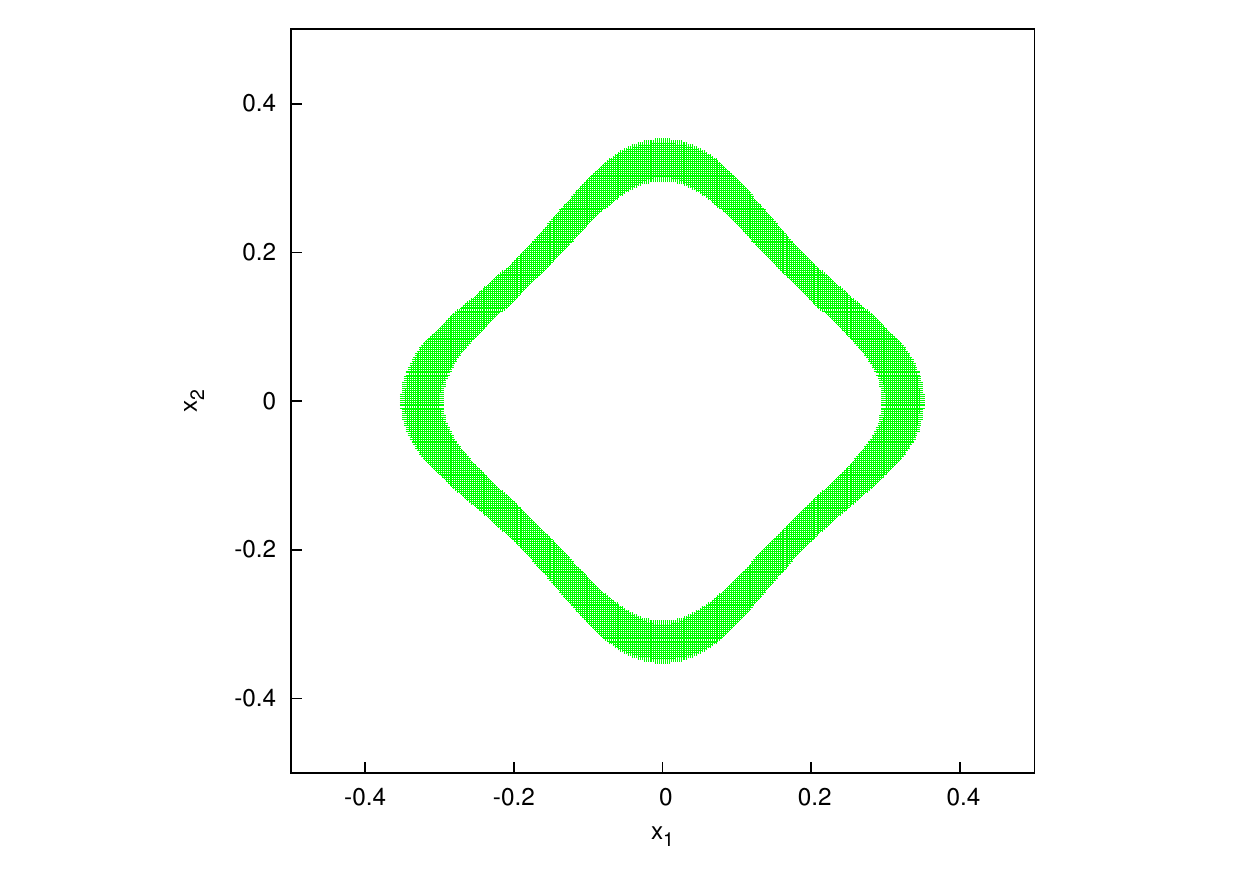}
\centering\includegraphics[scale=0.5]{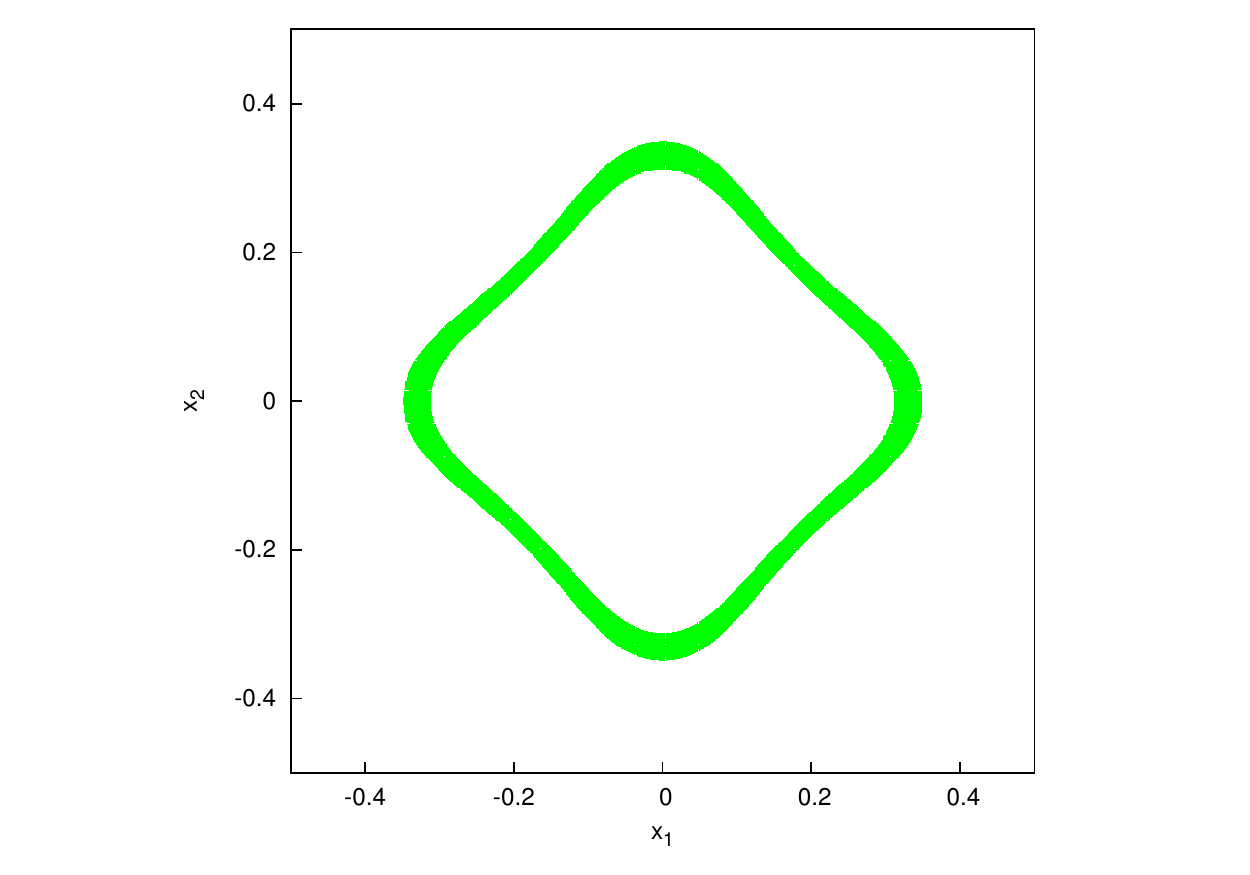}
\caption{Discrete coincidence set for Example 3 for level 7 (left) and level 8 (right).}
\label{fig:CoincidenceSetExample3}
\end{figure}
\par\smallskip
{\bf Example 4.}
 In this example we take $\O$ to be the $L$-shaped domain
 $(-0.5,0.5)^2\backslash[0,0.5]^2$, $f=g=0$ and
 $\psi(x)=1-\left[\frac{(x_1+0.25)^2}{0.2^2}+\frac{x_2^2}{0.35^2}\right].$
 We solve the discrete obstacle problems using a similar partition
 as described in Section~\ref{subsec:ApproximationSpace}.  For this example, $j$
 is chosen so that it is the level where there are $2^j+1$ subdivisions in each direction,
 making $h_j=(2^j+1-1/3)^{-1}.$  This allows us to insert an $L$-shaped element in
 the vicinity of the reentrant corner as described in Remark~\ref{rem:LShaped}.
\par
 From the numerical results in Table~\ref{Ex4} we observe that $\tilde{\beta}_h$ is approaching
 $\mathcal{O}(h^{\alpha})$ where $\alpha=0.544$ is the index of elliptic
 regularity for the $L$-shaped domain, as predicted by Theorem \ref{thm:Error}.
 \begin{table}[h]
 \begin{center}
 \begin{tabular}{ | c | c r | c r | }
 \hline &&&&\\[-12pt]
 $j$ & $\|\tilde{e}_j\|_j/\|u_{8}\|_{8} $ & $\tilde{\beta}_h$ &
 $\|\tilde{e}_j\|_{\infty} $ & $\tilde{\beta}_{\infty}$ \\ \hline &&&&\\[-12pt]
  1 &   4.4737   $\times10^{-0}$ &    &   1.0000   $\times10^{-0}$ &   \\ \hline &&&&\\[-12pt]
  2 &   6.9545   $\times10^{-0}$ &   -0.7884  &   5.9996   $\times10^{-1}$ & 0.9129   \\ \hline &&&&\\[-12pt]
  3 &   2.9079   $\times10^{-0}$ &   1.4086 &   3.2598   $\times10^{-1}$ &   0.9854 \\ \hline &&&&\\[-12pt]
  4 &   1.8562   $\times10^{-0}$ &   0.6864 &   1.3853   $\times10^{-1}$ &   1.3086 \\ \hline &&&&\\[-12pt]
  5 &   6.9086   $\times10^{-1}$ &   1.4687 &   4.0400   $\times10^{-2}$ &   1.8312 \\ \hline &&&&\\[-12pt]
  6 &   2.8930   $\times10^{-1}$ &   1.2747 &   2.9381   $\times10^{-2}$ &   0.4664 \\ \hline &&&&\\[-12pt]
  7 &   1.6919   $\times10^{-1}$ &   0.7797 &   1.4457   $\times10^{-2}$ &   1.0308 \\ \hline &&&&\\[-12pt]
  8 &   1.0582   $\times10^{-1}$ &   0.6796 &   6.9259   $\times10^{-3}$ &   1.0657 \\ \hline
   \end{tabular}
 \end{center}
\par\medskip
\caption{Energy norm and $\ell_{\infty}$ norm errors for Example 4.}\label{Ex4}
 \end{table}
\par
 Since $\Delta^2\psi-f=0$ for this example, the non-coincidence set is connected \cite{CF:1979:BiharmonicObstacle},
 which is confirmed by Figure~\ref{fig:CoincidenceSetExample4}.
\begin{figure}[h]
\centering\includegraphics[scale=0.5]{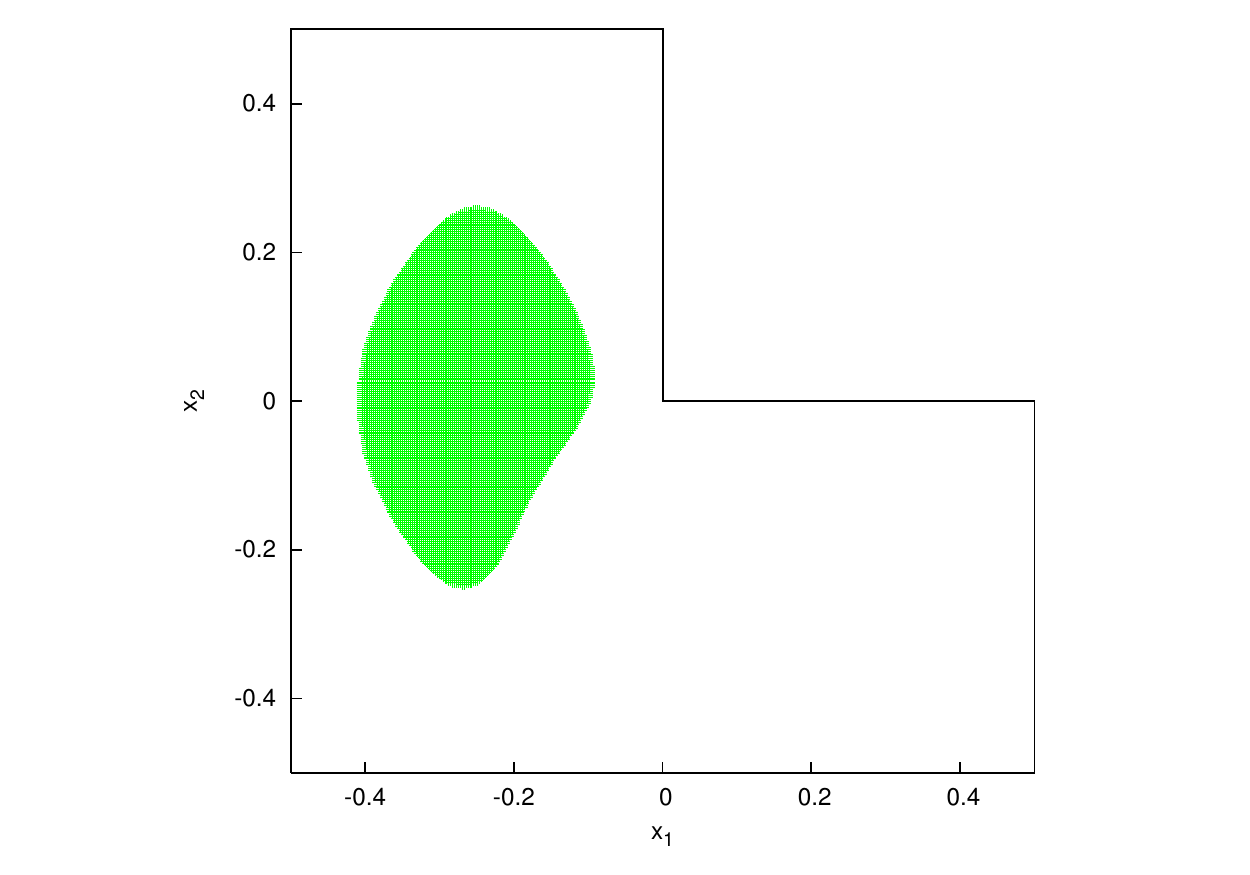}
\centering\includegraphics[scale=0.5]{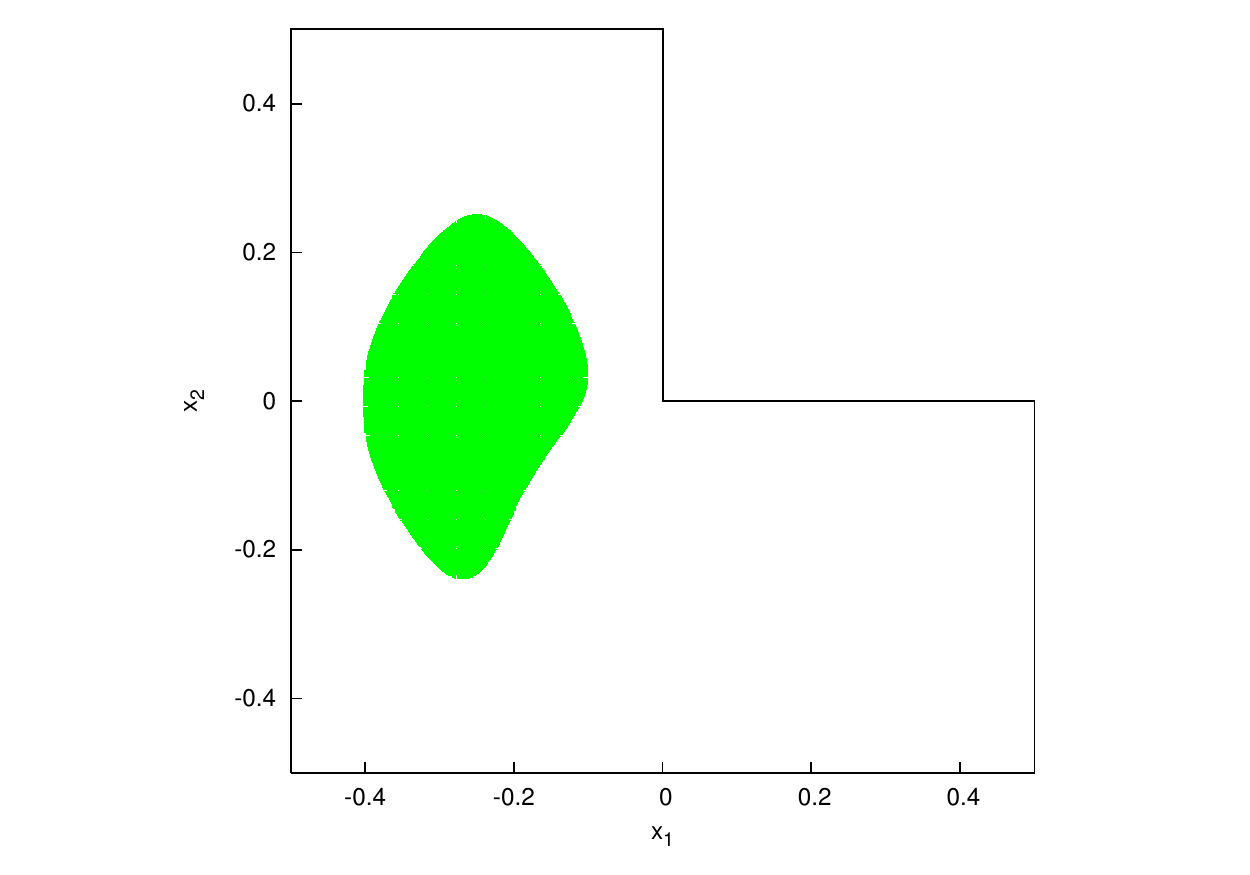}
\caption{Discrete coincidence set for Example 4 for level 7 (left) and level 8 (right).}
\label{fig:CoincidenceSetExample4}
\end{figure}
\par\smallskip
{\bf Example 5.}
 In this example we take $\O$ to be the pentagon $\{x\in(-0.5,0.5)^2:x_1+x_2<0.5\}$.
 We take $f=g=0$ and $\psi(x)=1-9|x|^2.$  We solve the discrete obstacle problems using
 a similar partition as described in Section~\ref{subsec:ApproximationSpace}.  For this example,
 $j$ is chosen so that it is the level where there are $2^j+1$ subdivisions
 in each direction, making $h_j=(2^j+1-1/3)^{-1}.$  This allows us to insert
 different types of elements near the obtuse vertices of $\O$, see Figure~\ref{fig:dofplacementPent}
 and Figure~\ref{fig:dofplacementPent2}.   The numerical results are reported in
 Table~\ref{Ex5}.
\begin{figure}[h]
\begin{center}
 \includegraphics[scale=.5]{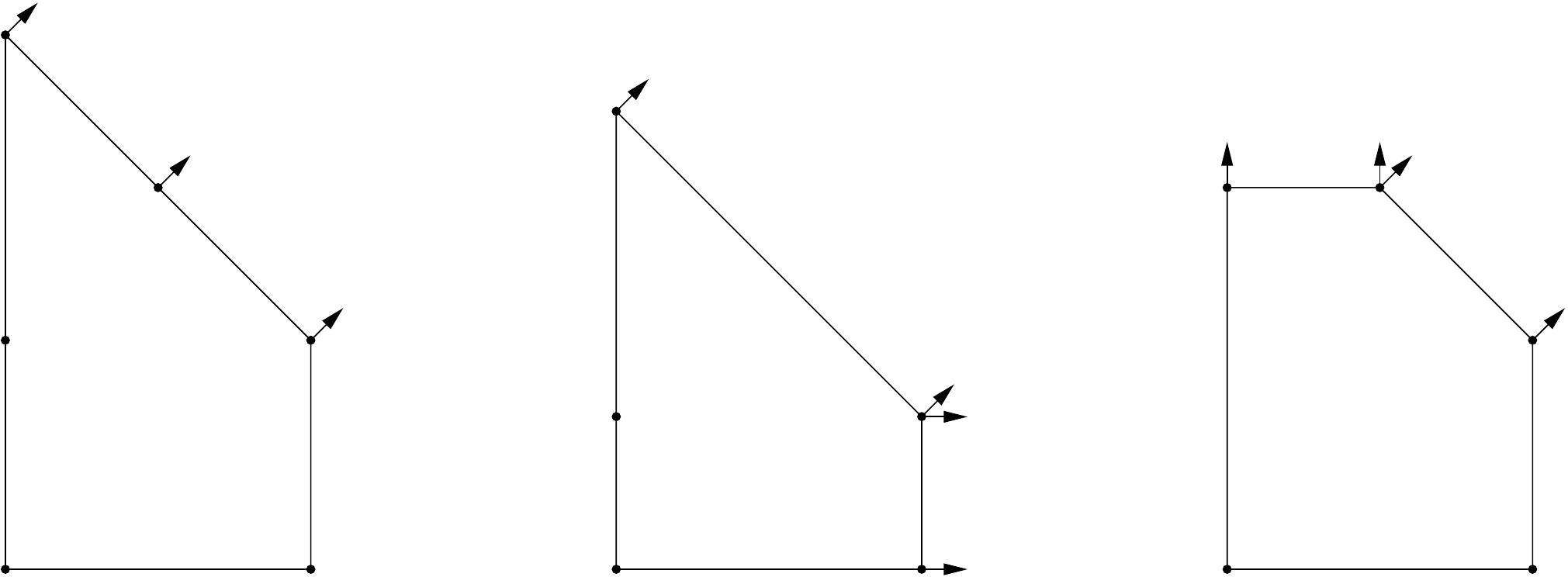}
 \caption{Reference elements for the pentagonal domain.}
\label{fig:dofplacementPent}
\end{center}
\end{figure}
\begin{figure}[h]
\begin{center}
 \includegraphics[scale=.75]{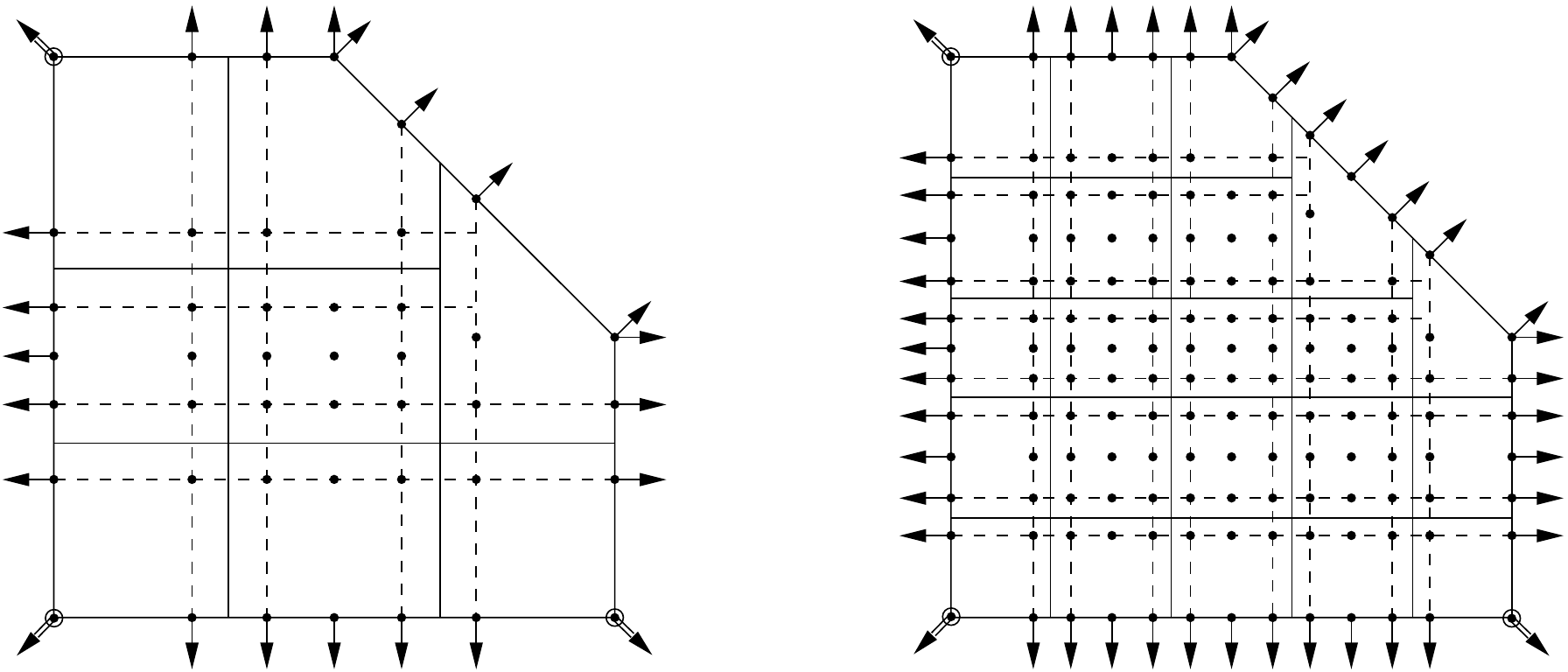}
 \caption{A partition of the pentagonal domain $\O$ for levels $j=1$ (left) and $j=2$ (right).  The
  solid lines separate the different patches $Q_j, j=1, \dots, 9$ (left)
  and $Q_j, j=1,\ldots,25$ (right).  The dashed lines represent
 the extension of $Q_j$ by $\delta(h/2)$ on each side.  This figure also shows the locations of the
 degrees of freedom.}
\label{fig:dofplacementPent2}
\end{center}
\end{figure}
 \begin{table}[h]
 \begin{center}
 \begin{tabular}{ | c | c r | c r | }
 \hline &&&&\\[-12pt]
 $j$ & $\|\tilde{e}_j\|_j/\|u_{8}\|_{8} $ & $\tilde{\beta}_h$ & $\|\tilde{e}_j\|_{\infty} $
 & $\tilde{\beta}_{\infty}$ \\ \hline &&&&\\[-12pt]
  1 &   5.1600   $\times 10^{-0}$ &    &   1.0628   $\times 10^{-0}$ &   \\ \hline &&&&\\[-12pt]
  2 &   1.0383   $\times 10^{+1}$ &   -1.2495  &   1.1233   $\times 10^{-0}$ & 0.0989   \\ \hline &&&&\\[-12pt]
  3 &   4.8834   $\times 10^{-0}$ &   1.2185 &   4.6694   $\times 10^{-1}$ &   1.4181 \\ \hline &&&&\\[-12pt]
  4 &   3.6378   $\times 10^{-0}$ &   0.4503 &   2.5049   $\times 10^{-1}$ &   0.9524 \\ \hline &&&&\\[-12pt]
  5 &   1.5514   $\times 10^{-0}$ &   1.2664 &   2.7118   $\times 10^{-2}$ &   3.3037 \\ \hline &&&&\\[-12pt]
  6 &   6.5449   $\times 10^{-1}$ &   1.2639 &   4.5364   $\times 10^{-3}$ &   2.6183 \\ \hline &&&&\\[-12pt]
  7 &   2.8868   $\times 10^{-1}$ &   1.1898 &   9.0815   $\times 10^{-4}$ &   2.3380 \\ \hline &&&&\\[-12pt]
  8 &   1.3864   $\times 10^{-1}$ &   1.0621 &   1.7631   $\times 10^{-4}$ &   2.3737 \\ \hline
   \end{tabular}
 \end{center}
\par\medskip
\caption{Energy norm and $\ell_{\infty}$ norm errors for Example 5.}\label{Ex5}
 \end{table}

\par
 Since $\Delta^2\psi -f=0$ in this example, the non-coincidence set is connected \cite{CF:1979:BiharmonicObstacle},
 which is confirmed by Figure~\ref{fig:CoincidenceSetEample5}, where the discrete coincidence sets also
 display the correct reflection symmetry.
\begin{figure}[h]
\centering\includegraphics[scale=0.5]{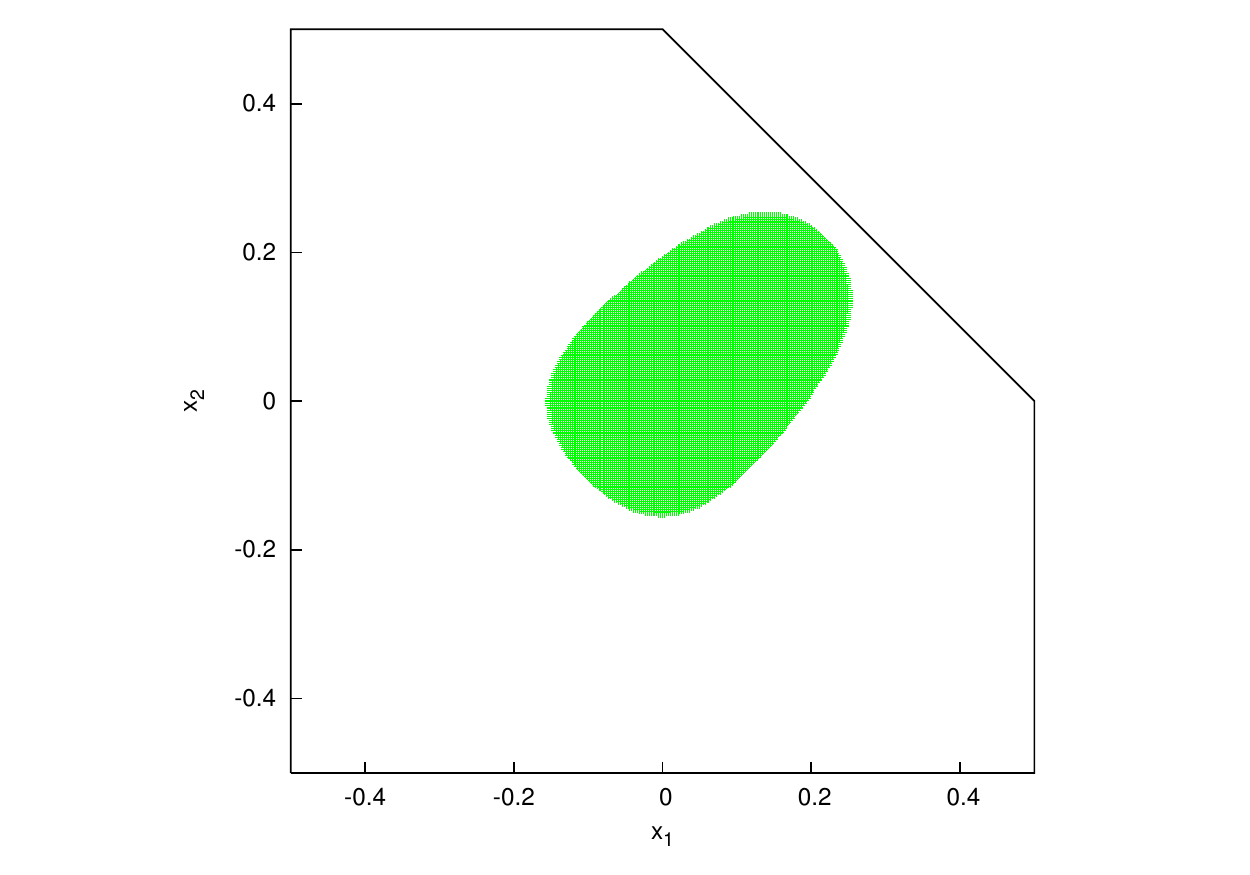}
\centering\includegraphics[scale=0.5]{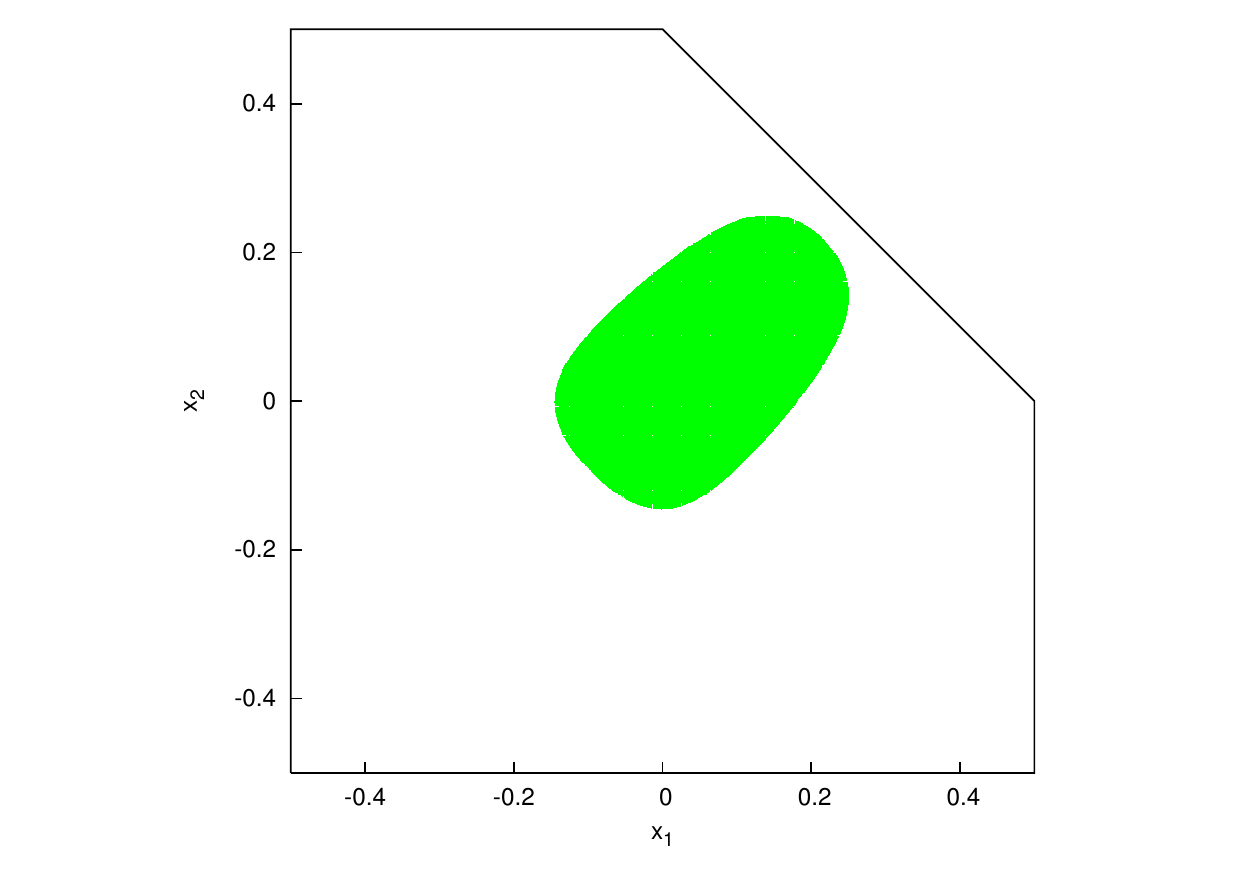}
\caption{Discrete coincidence set for Example 5 for level 7 (left) and level 8 (right).}
\label{fig:CoincidenceSetEample5}
\end{figure}

\end{document}